\documentclass{article}

\usepackage[utf8]{inputenc}
\usepackage[cyr]{aeguill}
\usepackage[square,sort,comma,numbers]{natbib}
\usepackage{url}
\usepackage{graphicx}
\usepackage{bm}
\usepackage[a4paper]{geometry}
\usepackage{amsmath,amsfonts,amsthm,amssymb}
\usepackage{stmaryrd}
\usepackage{color,xcolor}
\usepackage{subcaption}
\usepackage{array}
\usepackage{multirow}
\usepackage{url}
\usepackage{hyperref}
\usepackage{centernot}
\usepackage{tikz}
\usepackage[export]{adjustbox}

\newtheorem{example}{Example}
\newtheorem{rmk}{Remark}
\newtheorem{prop}{Proposition}

\newtheorem{definition}{Definition}

\newcommand{\bbu}{\mathbf{u}}
\newcommand{\bbv}{\mathbf{v}}
\newcommand{\bbw}{\mathbf{w}}
\newcommand{\bby}{\mathbf{y}}
\newcommand{\bbz}{\mathbf{z}}

\newcommand{\bbx}{\mathbf{x}}
\newcommand{\bbF}{\mathbf{F}}
\newcommand{\bbG}{\mathbf{G}}
\newcommand{\bbf}{\mathbf{f}}
\newcommand{\bbg}{\mathbf{g}}
\newcommand{\bbh}{\mathbf{h}}
\newcommand{\bbM}{\mathbf{M}}

\newcommand{\R}{\mathbb R}

\DeclareSymbolFont{matha}{OML}{txmi}{m}{it}
\DeclareMathSymbol{\varv}{\mathord}{matha}{118}
\begin{document}
\title{A positive- and bound-preserving vectorial lattice Boltzmann method in two dimensions}
\author{Gauthier Wissocq\footnote{corresponding author}, Yongle Liu, R\'emi Abgrall\\
Institute of Mathematics, University of Z\"urich, Switzerland\\
gauthier.wissocq@math.uzh.ch}
\date{}
\maketitle


\begin{abstract}
We present a novel positive kinetic scheme built on the efficient collide-and-stream algorithm of the lattice Boltzmann method (LBM) to address hyperbolic conservation laws. We focus on the compressible Euler equations with strong discontinuities. Starting from the work of Jin and Xin~\cite{Jin} and then \cite{Aregba,Bouchut}, we show how the LBM discretization procedure can yield both first- and second-order schemes, referred to as vectorial LBM. Noticing that the first-order scheme is convex preserving under a specific CFL constraint, we develop a blending strategy that preserves both the conservation and simplicity of the algorithm. This approach employs convex limiters, carefully designed to ensure either positivity (of the density and the internal energy) preservation (PP) or well-defined local maximum principles (LMP), while minimizing numerical dissipation. On challenging test cases involving strong discontinuities and near-vacuum regions, we demonstrate the scheme accuracy, robustness, and ability to capture sharp discontinuities with minimal numerical oscillations.
\end{abstract}

\textbf{\textit{Keywords---}} lattice Boltzmann method, hyperbolic conservation laws, convex limiting, positive preservation, kinetic method

\section{Introduction}
We are interested in the numerical approximation of systems of hyperbolic conservation laws in two dimensions
\begin{equation}
	\partial_t \bbu (\mathbf{x},t) + \partial_x \bbf(\bbu(\mathbf{x},t)) + \partial_y \bbg(\bbu(\mathbf{x},t)) = \mathbf{0}, \qquad \bbu(\mathbf{x},0) = \bbu_0(\mathbf{x}), \label{eq:hyperbolic_PDE}
\end{equation}
where $\bbu: \Omega \times \mathbb{R}^+ \mapsto \mathbb{R}^p$ is the vector of $p$ conserved variables, $\Omega \subset \mathbb{R}^2$ is a bounded smooth spatial domain and the fluxes $\bbf, \bbg:\mathbb{R}^p \rightarrow \mathbb{R}^p$ are Lipschitz continuous. 
Our main target is  the Euler equations for compressible flows, where
$\bbu=(\rho, \rho \bbv, E)^T \text{ and }(\bbf,\bbg)=(\rho \bbv, \rho \bbv\otimes\bbv+p\mathbf{I}_2, (E+p)\bbv)^T,$ where $\mathbf{I}_2$ is the $(2\times 2)$ identity matrix. The pressure $p$ is related to the conserved variable by an equation of state, and here we choose the ideal gas one $p=(\gamma-1)\big (E-\tfrac{1}{2}\rho\bbv^2\big)$, where $\gamma>1$ is a constant heat capacity ratio. To be admissible, the solution must belong to the convex set 
\begin{equation}
	\mathcal{G}=\{\bbu\in \R^4, \text{ such that }\rho>0 \text{ and } E-\tfrac{1}{2}\rho\bbv^2>0\}.
	\label{eq:convex_G}
\end{equation}
This system is equipped with an entropy $\eta$, a convex function of the conserved variables defined on $\mathcal{G}$. Any solution should satisfy the following inequality in the distributional sense
$\partial_t\eta+\text{div }\bbh \leq 0$ where the entropy flux is $\bbh=\eta \bbv$. For a perfect gas, the entropy is $\eta=-\rho\big (\log p-\gamma\log \rho\big ).$

Alternatively, weak solutions of the hyperbolic system \eqref{eq:hyperbolic_PDE} can be sought in the limit of a vanishing relaxation parameter of so-called advection-relaxation models. This is notably the principle of Jin-Xin model~\cite{Jin}, which was extended to the general framework of BGK kinetic models in the scalar case in~\cite{Natalini}, and for systems of PDEs in~\cite{Aregba}. Compared to a direct numerical resolution of the hyperbolic system \eqref{eq:hyperbolic_PDE}, the advantage of dealing with a kinetic advection-relaxation model is twofold: (1) the non-local terms are linear, avoiding the complex resolution of local Riemann problems, (2) under some conditions that are detailed in \cite{Bouchut} and that often reduce to Whitham's subcharacteristic condition~\cite{Whitham1974}, a kinetic model of the BGK type becomes compatible with entropy inequalities. 

On the other hand, a kinetic representation of a fluid through the Boltzmann equation has made it possible to build a simple and efficient numerical method for (weakly compressible) fluid mechanics problems, namely the lattice Boltzmann method (LBM)~\cite{McNamara1988, Kruger2017}. Its numerical scheme, also called \textit{collide and stream algorithm}, can be obtained through an appropriate time and space discretization of a discrete-velocity counterpart of the Boltzmann equation~\cite{He1998}. It is possible to apply the exactly same discretization procedure to any kinetic model in the framework of~\cite{Natalini, Aregba}, leading to efficient numerical methods consistent with hyperbolic systems such as compressible Euler: these  methods are referred to as vectorial LBM (VLBM). However, despite their promises and except few references, \textit{e.g.}~\cite{Graille2014, Dubois2014, Coulette2020, Caetano2023, Baty2023}, they have not gained significant attention from the LBM community for simulations of compressible flows. This is probably due to two reasons: (1) contrary to the standard LBM, these methods do not account for the viscous fluxes necessary to model finite Reynolds number flows, (2) as shown in~\cite{Graille2014}, the VLBM is subject to important oscillations close to discontinuities, which makes it not suitable for strongly compressible flows.
 
If the first issue has been dealt with in \cite{GauthierJCP1,GauthierJCP2}, the present paper addresses the second point. A particular problem caused by the numerical oscillations is the difficulty to preserve the solution in the convex domain $\mathcal{G}$. This is especially true for flows with  very strong discontinuities and near-vacuum regions. Convex preservation, and more precisely positive preservation (PP) for the Euler system using high-order schemes is a fundamental issue that raised the attention of many researchers, see e.g.~\cite{Perthame1996,Zhang2010,Guermond2018,Guermond2019,Vilar2019,Kuzmin2020,Hajduk2021,Wu2023,Vilar2024,Pampa2024} and the references within. A possible strategy is to blend a high-order scheme with  a PP first-order one  by convex limiting. Choosing well the parameters, the blended scheme can be made PP by construction. This strategy has been successfully applied in the context of the Euler equations to ensure  positivity of density, local/global maximum principles and pressure positivity~\cite{Vilar2019, Vilar2024, Duan2024-2D}. In the present work, we want to extend this strategy to the VLBM in order to make it PP and mostly oscillation free.

The format of this paper is as follows. First, we recall a few important facts about the VLBM and its formalism. In particular, we describe the collide and stream algorithm. Then, we show how to blend a first-order and a second-order schemes, under this collide and stream method, to get a conservative scheme that guaranties the positivity of the density and the internal energy under a known (reasonable) CFL condition. This is then illustrated by several examples, some quite demanding.
\section{Vectorial lattice Boltzmann schemes}

In the present work, we approximate the hyperbolic conservation laws \eqref{eq:hyperbolic_PDE} using  kinetic systems of the BGK type. The vector of conserved variables is defined as the zeroth-order moment of some distribution functions obeying advection-relaxation equations. Without loss of generality, this writes
\begin{equation}
	\partial_t \bbu_k^\varepsilon(\mathbf{x},t) + e_{x,k} \partial_x \bbu_k^\varepsilon(\mathbf{x}, t) + e_{y,k} \partial_y \bbu_k^\varepsilon(\mathbf{x}, t)  = \frac{1}{\varepsilon} \left( \bbM_k(\bbu^\varepsilon(\mathbf{x}, t))- \bbu_k^\varepsilon(\mathbf{x}, t) \right), \label{eq:kinetic}
\end{equation}
where $k=1,...,Q$, $\bbu_k^\varepsilon:\Omega \times \mathbb{R}^+ \rightarrow \mathbb{R}^p$ are the distribution functions, $e_{x,k}, e_{y,k} \in \mathbb{R}$ are kinetic velocities, $\varepsilon >0$ is a relaxation parameter, $\bbM_k:\mathbb{R}^p \rightarrow \mathbb{R}^p$ are Maxwellian functions, and $\bbu^\varepsilon (\mathbf{x}, t) := \sum_{k=1}^Q \bbu_k^\varepsilon (\mathbf{x}, t)$ is the zeroth moment. It can be shown \cite{Jin,Aregba,Natalini} that, when the following conditions are satisfied by the Maxwellian functions,
\begin{equation}
	\forall \bbu \in \mathbb{R}^p, \quad \sum_{k=1}^Q \bbM_k(\bbu) = \bbu, \quad \sum_{k=1}^Q e_{x,k} \bbM_k(\bbu) = \bbf(\bbu), \quad \sum_{k=1}^Q e_{y,k} \bbM_k(\bbu) = \bbg(\bbu), \label{eq:conditions_maxwellian} 
\end{equation}
then the solution $\bbu^\varepsilon$ of \eqref{eq:kinetic} formally tends to the solution $\bbu$ of the macroscopic system \eqref{eq:hyperbolic_PDE} when $\varepsilon \rightarrow 0$.

\begin{definition}[Bouchut criterion, from~\cite{Bouchut}]
	A set of Maxwellian functions $\bbM_k(\bbu)$ is said to be monotone non-decreasing if for all $k \in \{1, ..., Q\}$, $\bbM_k'(\bbu)$ is diagonalizable with nonnegative eigenvalues. 
	\end{definition}
	
This property \cite{Bouchut} guaranties that the system \eqref{eq:kinetic} admits an H-theorem if the hyperbolic system is equipped with a convex entropy. In the rest of this paper, the focus is put on a five-wave kinetic model (also referred to as D2Q5 using the notation of Qian \textit{et al.}~\cite{Qian1992}) whose details are provided in the example below. 


\begin{example}[Five-wave (D2Q5) model] \label{ex:D2Q5}
	We take $Q=5$ and consider the following kinetic velocities: $e_{x,k}=\{a, -a, 0, 0, 0\}$, $e_{y,k}=\{0, 0, a, -a, 0\}$ where $a>0$ is a kinetic speed. We choose Maxwellian functions as
\begin{equation*}
\begin{split}
 \bbM_1(\bbu) &= \left(\frac{1-\alpha}{4} \right) \bbu + \frac{\bbf(\bbu)}{2a}, \qquad\bbM_2(\bbu) = \left(\frac{1-\alpha}{4} \right) \bbu - \frac{\bbf(\bbu)}{2a}, \\
\bbM_3(\bbu) &= \left(\frac{1-\alpha}{4} \right) \bbu + \frac{\bbg(\bbu)}{2a}, \qquad\bbM_4(\bbu) = \left(\frac{1-\alpha}{4} \right) \bbu - \frac{\bbg(\bbu)}{2a}, \\
 \bbM_5(\bbu) &= \alpha \bbu,
\end{split}
\end{equation*}
where $\alpha \in [0, 1[$. These Maxwellian functions satisfy the conditions \eqref{eq:conditions_maxwellian}.
 \end{example}
 
 \begin{prop}
 	The Maxwellian functions of Example \ref{ex:D2Q5} are monotone nondecreasing (Bouchut criterion) if
 \begin{equation}
 	a\geq\frac{2}{1-\alpha} \max(\rho(\bbf'(\bbu)), \rho(\bbg'(\bbu))). \label{eq:Bouchut_D2Q5}
 \end{equation} 
 \end{prop}

 \begin{rmk}
 The four-wave model previously introduced in~\cite{Aregba} 
 is a particular case of our D2Q5 model for $\alpha=0$. In the present case, we have added a free parameter $\alpha$. Its role will become clearer in Section~\ref{sec:positive and BP}.
 \end{rmk}
 
In many references, \textit{e.g.} in~\cite{Graille2014, Dubois2014, Aregba-Driollet2024}, a (vectorial) lattice Boltzmann scheme can be considered by reinterpreting the kinetic system~\eqref{eq:kinetic} as a collision and a streaming phase on a lattice. The role played by the relaxation parameter in the collision is then investigated \textit{a posteriori}. In the present work, we propose to construct such a scheme following the \textit{a priori} construction of the LBM starting from a discrete-velocity Boltzmann equation by~\cite{He1998}, and we show how this can lead to both first- and second-order schemes. The two-dimensional domain $\Omega$ is discretized on a Cartesian mesh $\mathcal{C} = \{(x_i, y_j)\}$, where $x_{i+1}-x_i = y_{j+1}-y_j = \Delta x$. We also partition the time domain in intermediate times $\{t^n\}$ where $\Delta t^n = t^{n+1}-t^n$ is the $n$th time step. We start by integrating \eqref{eq:kinetic} along the characteristic lines $(x_i+e_{x,k}s, y_j + e_{y,k} s, t^n+s)$ for $s \in [0, \Delta t^n]$. This yields
 \begin{equation}
 \begin{split}
 	 \bbu_k^\varepsilon&(x_i+e_{x,k} \Delta t^n, y_j+e_{y,k} \Delta t^n, t^{n+1}) - \bbu_k^\varepsilon(x_i, y_j, t^n) =\\& \int_0^{\Delta t^n} \frac{1}{\varepsilon}\bigg[ \bbM_k(\bbu^\varepsilon(x_i+e_{x,k}s, y_i+e_{y,k}s, t+s)) 
	 - \bbu_k^\varepsilon (x_i+e_{x,k}s, y_i+e_{y,k}s, t+s) \bigg] \mathrm{d} s. \label{eq:integration_charac}
	\end{split}
 \end{equation}
 We first note that this leads to a numerical scheme in a finite-difference spirit provided that the points $(x_i + e_{x,k} \Delta t^n, y_j + e_{y,k} \Delta t^n) \in \mathcal{C}$ for all waves $1 \leq k \leq Q$. This implies a CFL=1 constraint related to the kinetic velocities. It is called  \textit{exact streaming} in the LBM community. This  yields the low dissipation of the method~\cite{Marie2009}. In the  case of Example~\ref{ex:D2Q5}, this condition reads $a^n {\Delta t^n}/{\Delta x} = 1$, where $a^n$ is the kinetic speed of the $n$th time step. 
 At first glance, this unitary CFL constraint may seem rude in the construction of a numerical method, as it does not leave us flexibility in the definition of the time step. In fact, we will see below that depending on the considered kinetic system and the choice of the kinetic speed, one still have freedom in the choice of the CFL number based on the largest eigenvalue of the hyperbolic system \eqref{eq:hyperbolic_PDE}.
 
 Until now, \eqref{eq:integration_charac} is exact. The next step is to approximate the  relaxation term. Since this is a stiff term, as $\varepsilon$ is likely to tend to zero, an implicit time integration has to be considered. In the following subsections, a first-order and a second-order implicit integrations of the relaxation term are proposed. Note that following appropriate variable changes inspired by what was initially proposed in~\cite{He1998}, it is possible to make both schemes fully explicit in time, as explained below.
 
 \subsection{First-order scheme}
 
 We approximate the right-hand-side term of \eqref{eq:integration_charac} using a first-order implicit Euler integration in time. Defining $\bbu_{k,i,j}^{\varepsilon,n} := \bbu_k^\varepsilon(x_i, y_j, t^n)$ and $\bbu_{i,j}^{\varepsilon,n} = \sum_k \bbu_{k,i,j}^{\varepsilon,n}$, this leads to
 \begin{equation}
 	\bbu_{k,i+\alpha_k,j+\beta_k}^{\varepsilon,n+1} - \bbu_{k,i,j}^{\varepsilon,n} = \frac{\Delta t^n}{\varepsilon} \left[ \bbM_k \left( \bbu_{i+\alpha_k,j+\beta_k}^{\varepsilon,n+1} \right) - \bbu_{k,i+\alpha_k, j+\beta_k}^{\varepsilon,n+1} \right], \label{eq:order1_implicit}
 \end{equation}
 where $\alpha_k = e_{x,k} \Delta t^n/\Delta x \in \mathbb{Z}$ and $\beta_k=e_{y,k} \Delta t^n/\Delta x \in \mathbb{Z}$. The scheme is implicit but can be made explicit by considering the variable change
 \begin{equation}
 	\overline{\bbu}_{k,i,j}^{\varepsilon,n} \stackrel{def}{=} \bbu_{k,i,j}^{\varepsilon, n} - \frac{\Delta t^n}{\varepsilon} \left[ \bbM_k(\bbu_{i,j}^{\varepsilon,n}) - \bbu_{k,i,j}^{\varepsilon,n} \right].\label{eq:variable_change_O1}
 \end{equation}
 With this new variable, Eq.~\eqref{eq:order1_implicit} reads
 \begin{equation}
 \begin{split}
 	\overline{\bbu}_{k,i+\alpha_k,j+\beta_k}^{\varepsilon, n+1} & = \overline{\bbu}_{k,i,j}^{\varepsilon, n} + \frac{\Delta t^n}{\varepsilon + \Delta t^n}  \left[ \bbM_k(\bbu_{i,j}^{\varepsilon,n}) - \overline{\bbu}_{k,i,j}^{\varepsilon, n} \right]. \label{eq:order1_explicit}
	\end{split}
 \end{equation}
A fundamental property of \eqref{eq:variable_change_O1} is that, since $\sum_k \left( \bbM_k(\bbu_{i,j}^{\varepsilon,n}) - \bbu_{k,i,j}^{\varepsilon,n} \right)=0$, then $ \sum_k \overline{\bbu}_{k,i,j}^{\varepsilon,n} = \sum_k {\bbu}_{k,i,j}^{\varepsilon, n} \stackrel{def}{=} \bbu_{i,j}^{\varepsilon, n}$. Hence, $\bbu^{\varepsilon}$ can be explicitly computed as the zeroth moment of $\overline{\bbu}_k^\varepsilon$: the scheme \eqref{eq:order1_explicit} is fully explicit in time. We finish  by letting $\varepsilon \rightarrow 0$ in order to make it consistent with the target hyperbolic system. Simply defining $\bbu_{k,i,j}^n := \lim_{\varepsilon \rightarrow 0} \overline{\bbu}_{k,i,j}^{\varepsilon, n}$ and $\bbu_{i,j}^n := \lim_{\varepsilon \rightarrow 0} \bbu_{i,j}^{\varepsilon, n}$, this yields
\begin{equation}
	\bbu_{k, i, j}^{n+1} = \bbM_k(\bbu_{i-\alpha_k,j-\beta_k}^n). \label{eq:1st-order_kinetic}
\end{equation}

\begin{example}[First-order scheme with the D2Q5 model]
\label{ex:D2Q5_O1}
Following the choices of Example~\ref{ex:D2Q5}, the first-order scheme reads
\begin{equation*}
\begin{split}
	\bbu_{1,i,j}^{n+1} &= \bbM_1(\bbu_{i-1, j}^n), \quad \bbu_{2, i, j}^{n+1} = \bbM_2(\bbu_{i+1,j}^n), \\
	\bbu_{3,i,j}^{n+1} &= \bbM_3(\bbu_{i,j-1}^n), \quad \bbu_{4,i,j}^{n+1} = \bbM_4(\bbu_{i,j+1}^n), \quad \bbu_{5,i,j}^{n+1} = \bbM_5(\bbu_{i,j}^n),
	\end{split}
\end{equation*}
and the conserved variables are finally updated with $\bbu_{i,j}^{n+1}=\sum_k \bbu_{k,i,j}^{n+1}$.
\end{example}
We finish this subsection by some important properties satisfied by this first-order kinetic scheme.

\begin{prop}
	In the case of scalar conservation laws ($p=1$), the first-order scheme is monotone if the Maxwelian functions are nondecreasing functions of their argument (Bouchut criterion~\cite{Bouchut}).
\end{prop}
\begin{proof}
	In the scalar case, we note $\bbu_{i,j}^n = [\rho_{i,j}^n]$ and $\bbM_k(\bbu_{i,j}^n) = [M_k(\rho_{i,j}^n)]$. The first-order scheme reads $\rho_{i,j}^{n+1} = \sum_k M_k(\rho_{i-\alpha_k, j-\beta_k}^n)$ and since each $M_k$ is a nondecreasing function, then the update function of the scheme is nondecreasing in each of its argument. Following~\cite{Harten1976}, the scheme is monotone.
\end{proof}

\begin{prop}
	\label{prop:convex_preservation_O1}
	Consider the first-order scheme of the D2Q5 model (Example~\ref{ex:D2Q5_O1}), and assume that the solution at time $t^n$ is in the convex $\mathcal{G}$. Then, the updated solution is in $\mathcal{G}$ under the constraint  $1/2 \leq \alpha < 1$. This yields a maximal CFL condition
	\begin{equation*}
		\lambda^n \frac{\Delta t^n}{\Delta x} \leq \frac{1}{4},
	\text{ where }
		\lambda^n := \max_{i,j}(\sigma(\bbf'(\bbu_{i,j}^n), \sigma(\bbg'(\bbu_{i,j}^n)). 	
	\end{equation*}
\end{prop}
\begin{proof}
	We start by developing the Maxwellian  of the D2Q5 model in the update of the conserved variables with the first-order scheme:
	\begin{equation}\label{eq:D2Q5_convexproof_1}
    \begin{split}
		\bbu_{i,j}^{n+1} &= \bbM_1(\bbu_{i-1,j}^n) + \bbM_2(\bbu_{i+1,j}^n) + \bbM_3(\bbu_{i,j-1}^n) + \bbM_4(\bbu_{i,j+1}^n) + \bbM_5(\bbu_{i,j}^n) \\
		& = \alpha \bbu_{i,j}^n + \left( \frac{1-\alpha}{4} \right) \left(\bbu_{i-1,j}^n + \bbu_{i+1,j}^n + \bbu_{i,j-1}^n + \bbu_{i,j+1}^n\right)   \\
		& \qquad + \frac{\bbf(\bbu_{i-1,j}^n)-\bbf(\bbu_{i+1,j}^n) + \bbg(\bbu_{i,j-1}^n) - \bbg(\bbu_{i,j+1}^n)}{2a^n} 
        \\
		& = (2\alpha-1) \bbu_{i,j}^n + \left(\frac{1-\alpha}{4} \right) \left[ \bbu_{i-1,j}^n + \bbu_{i+1,j}^n + \bbu_{i,j-1}^n + \bbu_{i,j+1}^n + 4 \bbu_{i,j}^n \right] \\
		& \qquad - \frac{1}{2a^n} \left[ \bbf(\bbu_{i+1,j}^n) - \bbf(\bbu_{i,j}^n) + \bbf(\bbu_{i,j}^n) - \bbf(\bbu_{i-1,j}^n) \right]  \\
		& \qquad - \frac{1}{2a^n} \left[ \bbg(\bbu_{i,j+1}^n) - \bbg(\bbu_{i,j}^n) + \bbg(\bbu_{i,j}^n) - \bbg(\bbu_{i,j-1}^n) \right]. 
        \end{split}
	\end{equation}
	We then write the kinetic speed as $a^n = 2\beta \lambda^n/(1-\alpha)$ with $\beta \geq 1$ to satisfy Bouchut's criterion \eqref{eq:Bouchut_D2Q5} and define the following intermediate states
	\begin{equation}
	\begin{split}
		 \bbu_{i+1/2,j}^* &:= \frac{1}{2} \left( \bbu_{i+1,j}^n + \bbu_{i,j}^n \right) - \frac{1}{2\beta \lambda^n} \left( \bbf(\bbu_{i+1,j}^n) - \bbf(\bbu_{i,j}^n) \right), \\
		 \bbu_{i,j+1/2}^* &:= \frac{1}{2} \left( \bbu_{i,j+1}^n + \bbu_{i,j}^n \right) - \frac{1}{2 \beta \lambda^n} \left( \bbg(\bbu_{i,j+1}^n) - \bbg(\bbu_{i,j}^n) \right).
		\end{split} \label{eq:D2Q5_convexproof_3}
	\end{equation}
	Eq.~\eqref{eq:D2Q5_convexproof_1} can be written as function of these intermediate states as
	\begin{equation}
		\bbu_{i,j}^{n+1} = (2\alpha-1) \bbu_{i,j}^n + \left( \frac{1-\alpha}{2} \right) \left[ \bbu_{i+1/2,j}^* + \bbu_{i-1/2,j}^* + \bbu_{i,j+1/2}^* + \bbu_{i,j-1/2}^* \right]. \label{eq:D2Q5_convexproof_2}
	\end{equation}
	To get $\bbu_{i,j}^{n+1} \in \mathcal{G}$, it is sufficient to show that this decomposition is a convex decomposition of elements in $\mathcal{G}$.
	First, note that if the solution at time $t^n$ is in $\mathcal{G}$, then $\bbu_{i+1/2,j}^* \in \mathcal{G}$ thanks to the following convex decomposition
	\begin{equation*}
		\bbu_{i+1/2,j}^* = \frac{1}{\beta} \overline{\bbu}_{i+1/2,j} + \frac{\beta-1}{2 \beta} (\bbu_{i+1,j}^n + \bbu_{i,j}^n),
	\end{equation*} 
	where $\overline{\bbu}_{i+1/2,j} = (\bbu_{i+1,j}^n + \bbu_{i,j}^n)/2 - (\bbf(\bbu_{i+1,j}^n) - \bbf(\bbu_{i,j}^n))/(2\lambda^n)$ is the Riemann intermediate state, which is clearly in $\mathcal{G}$ (see \textit{e.g.}~\cite{Guermond2016}). The same stands for $\bbu_{i,j+1/2}^n$. Hence, Eq.~\eqref{eq:D2Q5_convexproof_2} is a convex decomposition of elements in $\mathcal{G}$ provided that $1/2 \leq \alpha < 1$. This leads to the following condition on the CFL number based on $\lambda^n$:
	\begin{equation*}
		\lambda^n \frac{\Delta t^n}{\Delta x} = \frac{a^n (1-\alpha)}{2\beta} \frac{\Delta t^n}{\Delta x} =  \frac{1-\alpha}{2\beta} \leq \frac{1}{4\beta} \leq 1/4.
	\end{equation*}
	\end{proof}	

\begin{prop}
Considering the D2Q5 model of Example~\ref{ex:D2Q5}, the first-order lattice Boltzmann scheme \eqref{eq:1st-order_kinetic} is equivalent to the following finite-volume scheme,
\begin{equation}
\bbu_{i,j}^{n+1} = \bbu_{i,j}^n - \frac{\Delta t^n}{\Delta x} \left( \bbF_{i+1/2,j}^{(1),n} - \bbF_{i-1/2,j}^{(1),n} + \bbG_{i,j+1/2}^{(1),n} - \bbG_{i,j-1/2}^{(1),n} \right), \label{eq:1st-order_FV}
\end{equation}
with the fluxes
\begin{equation}\label{eq:1st-order_flux}
\begin{split}
	 \bbF_{i+1/2,j}^{(1),n}& = \frac{1}{2} \left( \bbf(\bbu_{i+1,j}^n)+\bbf(\bbu_{i,j}^n) \right) - a^n \left( \frac{1-\alpha}{4} \right) \left(\bbu_{i+1,j}^n - \bbu_{i,j}^n\right), \\
	\bbG_{i,j+1/2}^{(1),n} &=\frac{1}{2} \left( \bbg(\bbu_{i,j+1}^n)+\bbg(\bbu_{i,j}^n) \right) - a^n \left( \frac{1-\alpha}{4} \right) \left(\bbu_{i,j+1}^n - \bbu_{i,j}^n\right).
	\end{split}
\end{equation}
With the strict subcharacteristic condition $a^n =2\lambda^n/(1-\alpha)$, it is equivalent to the global Lax-Friedrichs scheme at CFL number equal to $(1-\alpha)/2$.
\end{prop}
\begin{proof}
	We start with Eq.~\eqref{eq:D2Q5_convexproof_1} and using  $a^n = \Delta x/\Delta t^n$,  the scheme  can be put under the form of Eq.~\eqref{eq:1st-order_FV} with the fluxes \eqref{eq:1st-order_flux}. 	With $a^n=2\lambda^n/(1-\alpha)$, we recognize the global Lax-Friedrichs fluxes in two dimensions. The CFL number based on $\lambda^n$ is $\lambda^n \Delta t^n/\Delta x = a^n \Delta t^n(1-\alpha)/(2\Delta x)=(1-\alpha)/2$.
\end{proof}


 \subsection{Second-order scheme}
 
 To build a second-order accurate method in space and time, we approximate the right-hand-side term of Eq.~\eqref{eq:integration_charac} with the trapezium rule (equivalent to the Crank-Nicholson scheme). This leads to
 \begin{equation}
 \begin{split}
 	\bbu_{k,i+\alpha_k, j+\beta_k}^{\varepsilon, n+1} &- \bbu_{k,i,j}^{\varepsilon, n} \\= \frac{\Delta t^n}{2 \varepsilon} \bigg[ &\bbM_k \left( \bbu_{i,j}^{\varepsilon, n} \right) - \bbu_{k,i,j}^{\varepsilon, n} + \bbM_k \left( \bbu_{i+\alpha_k,j+\beta_k}^{\varepsilon,n+1} \right) - \bbu_{k,i+\alpha_k, j+\beta_k}^{\varepsilon,n+1} \bigg].\label{eq:order2_implicit}
	\end{split}
 \end{equation}
 As in the construction of the first-order scheme and following~\cite{He1998}, this scheme is implicit but can be made explicit by considering the new variable
 \begin{equation}
 	\overline{\bbu}_{k,i,j}^{\varepsilon, n} \stackrel{def}{=} \bbu_{k,i,j}^{\varepsilon, n} - \frac{\Delta t^n}{2\varepsilon} \left[\bbM_k(\bbu_{i,j}^{\varepsilon, n}) - \bbu_{k,i,j}^{\varepsilon, n} \right].\label{eq:variable_change_O2}
 \end{equation}
 Equation \eqref{eq:order2_implicit} then reads
 \begin{equation*}
 \begin{split}
 	\overline{\bbu}_{k,i+\alpha_k, j+\beta_k}^{\varepsilon, n+1} & = \overline{\bbu}_{k,i,j}^{\varepsilon, n} + \frac{\Delta t^n}{\varepsilon} \left[ \bbM_k(\bbu_{i,j}^{\varepsilon, n}) - \bbu_{k,i,j}^{\varepsilon, n} \right] \\
 	& = \overline{\bbu}_{k,i,j}^{\varepsilon, n} + \frac{\Delta t}{\varepsilon + \Delta t/2} \left[ \bbM_k(\bbu_{i,j}^{\varepsilon, n}) - \overline{\bbu}_{k,i,j}^{\varepsilon, n} \right],
	\end{split}
 \end{equation*} 
 where 
 \eqref{eq:variable_change_O2} has been used to obtain the second equation. 
 The change of variable \eqref{eq:variable_change_O2} preserves the zeroth-order moment so that $\bbu_{i,j}^{\varepsilon, n} = \sum_k \overline{\bbu}_{k,i,j}^{\varepsilon, n}$. Hence, the scheme is explicit. We finish the construction by letting $\varepsilon \rightarrow 0$ which, eventually removing the $\varepsilon$ exponent and the $\overline{\cdot}$ for the sake of simplicity, leads to
 \begin{equation}
 	\bbu_{k,i,j}^{n+1} = 2 \bbM_k \left(\bbu_{i-\alpha_k, j-\beta_k}^{n} \right) - \bbu_{k,i-\alpha_k, j-\beta_k}^{n}. \label{eq:2nd-order_kinetic}
 \end{equation}
 
 \begin{example}[Second-order scheme with the D2Q5 model]
 	\label{ex:D2Q5_order2}
	Following the choices of Example~\ref{ex:D2Q5}, the second-order scheme reads
	\begin{equation*}
	\begin{split}
		 \bbu_{1,i,j}^{n+1} &= 2 \bbM_1(\bbu_{i-1,j}^n) - \bbu_{1,i-1,j}^n, \qquad \bbu_{2,i,j}^{n+1} = 2\bbM_2(\bbu_{i+1,j}^n) - \bbu_{2,i+1,j}^n, \\
		 \bbu_{3,i,j}^{n+1} &= 2 \bbM_3 (\bbu_{i,j-1}^n) - \bbu_{3,i,j-1}^n, \qquad \bbu_{4,i,j}^{n+1} = 2 \bbM_4(\bbu_{i,j+1}^n) - \bbu_{4,i,j+1}^n, \\
		\bbu_{5,i,j}^{n+1} &= 2 \bbM_5(\bbu_{i,j}^n) - \bbu_{5,i,j}^n.
		\end{split}
	\end{equation*}
\end{example}

\begin{prop}
	With the D2Q5 model of Example~\ref{ex:D2Q5}, the second-order sche\-me \eqref{eq:2nd-order_kinetic} is consistent with the target hyperbolic system \eqref{eq:hyperbolic_PDE}.	
\end{prop}
\begin{proof}
	Inspired by the work of Bellotti \textit{et al.}~\cite{Bellotti2022}, we can show (see Appendix~\ref{appendixA}) 
    that the second-order scheme can be reinterpreted as the following multi-step finite difference scheme over the conserved variables:
	\begin{equation}
	\begin{split}
	& \bbu_{i,j}^{n+5} + (4A_a-1) (\bbu_{i,j}^{n+4}-\bbu_{i,j}^{n+1}) + (4A_d-4A_a + 2) (\bbu_{i,j}^{n+3}-\bbu_{i,j}^{n+2}) - \bbu_{i,j}^n \label{eq:multistep_FD} \\
	& + 2 \left[ D_x \bbf(\bbu_{i,j}^{n+4})/a_{n+4} + 3 \overline{D}_x \bbf(\bbu_{i,j}^{n+3})/a_{n+3} + 3 \overline{D}_x \bbf(\bbu_{i,j}^{n+2})/a_{n+2} + D_x \bbf(\bbu_{i,j}^{n+1})/a_{n+1} \right]  \\
	& + 2 \left[ D_y \bbg(\bbu_{i,j}^{n+4})/a_{n+4} + 3 \overline{D}_y \bbg(\bbu_{i,j}^{n+3})/a_{n+3} + 3 \overline{D}_y \bbg(\bbu_{i,j}^{n+2})/a_{n+2} + D_y \bbg(\bbu_{i,j}^{n+1})/a_{n+1} \right] \\
	& + 2(1-\alpha) \left[ (1-A_a) (\bbu_{i,j}^{n+4}-\bbu_{i,j}^{n+1}) + (3A_a-2A_d-1)(\bbu_{i,j}^{n+3}-\bbu_{i,j}^{n+2}) \right] = \mathbf{0}, 
	\end{split}
	\end{equation}
	where
	\begin{equation}\label{eq:A_ad}
	\begin{split}
	A_a \bbu_{i,j} &= \frac{1}{4} \left( \bbu_{i+1,j} + \bbu_{i-1,j} + \bbu_{i,j+1} + \bbu_{i,j-1} \right), 
	 \\
	A_d \bbu_{i,j} &= \frac{1}{4} \left( \bbu_{i+1,j+1} + \bbu_{i-1,j+1} + \bbu_{i-1,j+1} + \bbu_{i-1,j-1} \right),
	\\
 D_x \bbf(\bbu_{i,j})& = \frac{1}{2} (\bbf(\bbu_{i+1,j}) - \bbf(\bbu_{i-1,j})), \\
	\overline{D}_x \bbf(\bbu_{i,j})& = \frac{1}{3} \left( D_x \bbf(\bbu_{i,j-1}) + D_x \bbf(\bbu_{i,j}) + D_x \bbf(\bbu_{i,j+1})\right), \\
	 D_y \bbg(\bbu_{i,j})& = \frac{1}{2} (\bbg(\bbu_{i,j+1}) - \bbg(\bbu_{i,j-1})), \\
	 \overline{D}_y \bbg(\bbu_{i,j}) &= \frac{1}{3} \left( D_y \bbg(\bbu_{i-1,j}) + D_y \bbg(\bbu_{i,j}) + D_y \bbg(\bbu_{i+1,j}) \right).
	\end{split}
	\end{equation}
	We then notice that $a^{n+1} = a^n + \mathcal{O}(\Delta t^n) = \Delta t^n/\Delta x + \mathcal{O}(\Delta t^n)$, and similarly for $a_{n+2}$, $a_{n+3}$ and $a_{n+4}$. We finally let $\Delta t^n, \Delta x \rightarrow 0$ in \eqref{eq:multistep_FD} and observe that it is consistent with \eqref{eq:hyperbolic_PDE}.
\end{proof}

\begin{prop}
	The second-order kinetic scheme \eqref{eq:2nd-order_kinetic} can be equivalently written under the following finite-volume form
	\begin{equation}
		\bbu_{i,j}^{n+1} = \bbu_{i,j}^n - \frac{\Delta t^n}{\Delta x} \left( \bbF_{i+1/2,j}^{(2),n} - \bbF_{i-1/2,j}^{(2),n} + \bbG_{i,j+1/2}^{(2),n} - \bbG_{i,j-1/2}^{(2),n} \right) \label{eq:2nd-order_FV}, 
	\end{equation}	
	with the fluxes
	\begin{equation}\label{eq:2nd-order_flux}
	\begin{split}
		 \bbF_{i+1/2,j}^{(2),n} &= \bbF_{i+1/2,j}^{(1),n} + \Delta \bbF_{i+1/2,j}^n, \ 
		 \bbG_{i,j+1/2}^{(2),n} = \bbG_{i,j+1/2}^{(1),n} + \Delta \bbG_{i,j+1/2}^n, 
		\\
		 \Delta \bbF_{i+1/2,j}^n &= a^n \left[ \bbM_1(\bbu_{i,j}^n) - \bbu_{1,i,j}^n - \bbM_2(\bbu_{i+1,j}^n) + \bbu_{2,i+1,j}^n \right], \\
		 \Delta \bbG_{i,j+1/2}^n& = a^n \left[ \bbM_3(\bbu_{i,j}^n) - \bbu_{3,i,j}^n - \bbM_4(\bbu_{i,j+1}^n) + \bbu_{4,i,j+1}^n \right].
		\end{split}
	\end{equation}
\end{prop}
\begin{proof}
	From the scheme provided in Example~\ref{ex:D2Q5_order2}, we compute
	\begin{equation*}
	\begin{split}
		\bbu_{i,j}^{n+1} & = \sum_{k=1}^5 \bbu_{k,i,j}^{n+1}  = \bbu_{i,j}^{(1),n+1} + \bbM_1(\bbu_{i-1,j}^n) - \bbu_{1,i-1,j}^n + \bbM_2(\bbu_{i+1,j}^n) - \bbu_{2,i+1,j}^n \nonumber \\
		& + \bbM_3(\bbu_{i,j-1}^n) - \bbu_{3,i,j-1}^n + \bbM_4(\bbu_{i,j+1}^n) - \bbu_{4,i,j+1}^n + \bbM_5(\bbu_{i,j}^n) - \bbu_{5,i,j}^n,
		\end{split}
	\end{equation*}
	where $\bbu_{i,j}^{(1),n+1}$ is the update of the conserved variables that would be obtained with the first-order scheme. Then noticing that $\bbu_{5,i,j}^n = \bbu_{i,j}^n - \sum_{k=1}^4 \bbu_{k,i,j}^n$ and that $\bbM_5(\bbu_{i,j}^n) = \bbu_{i,j}^n - \sum_{k=1}^4 \bbM_k(\bbu_{i,j}^n)$, we have
	\begin{equation*}
	\begin{split}
		\bbu_{i,j}^{n+1} = \bbu_{i,j}^{(1),n+1} &+ \bbM_1(\bbu_{i-1,j}^n) - \bbM_1(\bbu_{i,j}^n) - \bbu_{1,i-1,j}^n + \bbu_{1,i,j}^n  \\
		& + \bbM_2(\bbu_{i+1,j}^n) - \bbM_2(\bbu_{i,j}^n) - \bbu_{2,i+1,j}^n + \bbu_{2,i,j}^n \\
		& + \bbM_3(\bbu_{i,j-1}^n) - \bbM_3(\bbu_{i,j}^n) - \bbu_{3,i,j-1}^n + \bbu_{3,ij}^n  \\
		& + \bbM_4(\bbu_{i,j+1}^n) - \bbM_4(\bbu_{i,j}^n) - \bbu_{4,i,j+1}^n + \bbu_{4,i,j}^n,
		\end{split}
	\end{equation*}
	which can be put under the form \eqref{eq:2nd-order_FV} with the fluxes \eqref{eq:2nd-order_flux}.
\end{proof}


\begin{rmk}
Investigating the linear stability of the second-order scheme can be done by a von Neumann analysis, see for example~\cite{Sterling1996}. In this context, it was shown in~\cite{Graille2014} that a two-velocity (D1Q2) LBM is $L^2$-stable in the linear scalar case. More recently, a similar proof was provided for the D1Q3 lattice in~\cite{Bellotti2024}. However, proving the von Neumann stability of the LBM with the D2Q5 model considered in this work is a more tedious task that we leave for future investigation.
\end{rmk}

\section{Positive- and bound-preserving blended scheme}
\label{sec:positive and BP}

\subsection{Problem statement}

As shown for example in~\cite{Graille2014}, the second-order vectorial LB scheme is subject to strong oscillations close to discontinuities. This makes it unsuitable for compressible flows. To reduce these oscillations and ensure convex preservation properties of the vectorial lattice Boltzmann scheme, we propose a blending between the second-order scheme \eqref{eq:2nd-order_kinetic} and the first-order scheme \eqref{eq:1st-order_kinetic}. The idea is to introduce a new local parameter $\theta$ which dynamically varies between $0$ and $1$. When $\theta=0$, the blended scheme reduces to the first-order, while when $\theta=1$, it reduces to the second-order. A first possibility to introduce a blending is to affect one value of $\theta$ in each cell of the mesh, and build a blended scheme like
\begin{equation}
	\bbu_{k,i,j}^{n+1} = \bbM_k(\bbu_{i-\alpha_k, j-\beta_k}^n) + \theta_{i-\alpha_k, j-\beta_k}^n \left[ \bbM_k(\bbu_{i-\alpha_k, j-\beta_k}^n) - \bbu_{k,i-\alpha_k, j-\alpha_k}^n \right], \label{eq:naive_blending}
\end{equation}
where $\theta_{i,j}^n$ is a blending parameter that depends on the known solution at time $t_n$ to be defined. Note that the blending of \eqref{eq:naive_blending} follows the simple \textit{collide and stream} algorithm of the standard LBM~\cite{Kruger2017}: it can be decomposed into a local ``collision'' (right-hand-side term), followed by an exact transport phase. As such, as demonstrated in~\cite{Wissocq2022} (see Appendix B), the zeroth moment $\bbu_{i,j}=\sum_k \bbu_{k,i,j}$ of \eqref{eq:naive_blending} is a conserved variables despite the introduction of $\theta_{i,j}^n$. Note for example that the artificial viscosity introduced in~\cite{Wissocq2023} can be interpreted as a blending like the one in \eqref{eq:naive_blending}. In a similar fashion, the a posteriori correction developed in~\cite{Kozhanova2024} is a blending of this type with a blending parameter $\theta_{i,j}^n \in \{0, 1\}$.

In the present work, we want to introduce a blending parameter acting as a flux limiter, switching from the first-order fluxes \eqref{eq:1st-order_flux} to the second-order ones \eqref{eq:2nd-order_flux}. The blending of \eqref{eq:naive_blending} does not satisfy such property.  
Therefore, this question is addressed in the section below.
  For the sake of simplicity, we now restrict ourselves to the D2Q5 model of Example~\ref{ex:D2Q5}, but the method is more general.

\subsection{Construction of the blended scheme}

The idea behind the construction of the proposed blended schemes lies on the bijection between distributions and moments. With the D2Q5 lattice, we can define a moments matrix $\mathbf{Q}$ as in \eqref{eq:moments_matrix} with which we define moments as
\begin{equation}
	[\bbu_{i,j}^n, \bbv_{1,i,j}^n, \bbv_{2,i,j}^n, \bbw_{1,i,j}^n, \bbw_{2,i,j}^n]= \mathbf{Q} [\bbu_{1,i,j}^n, \bbu_{2,i,j}^n, \bbu_{3,i,j}^n, \bbu_{4,i,j}^n, \bbu_{5,i,j}^n]^T.
\end{equation} 
Note that the choice of the matrix $\mathbf{Q}$ is arbitrary, we only need that the first row of $\mathbf{Q}$ yields the computation of the zeroth moment and that $\mathbf{Q}$ is invertible. With these definitions, the first-order scheme reads
\begin{equation*}
\begin{split}
	 \bbu_{i,j}^{n+1}& = \bbu_{i,j}^n - \frac{\Delta t^n}{\Delta x} \left[ \bbF_{i+1/2,j}^{(1),n} - \bbF_{i-1/2,j}^{(1),n}  \right] - \frac{\Delta t^n}{\Delta x} \left[ \bbG_{i,j+1/2}^{(1),n} - \bbG_{i,j-1/2}^{(1),n}  \right], \\
	 \bbv_{1,i,j}^{n+1}& = \bbM_1(\bbu_{i-1,j}^n) - \bbM_2(\bbu_{i+1,j}^n), 
	\quad \bbv_{2,i,j}^{n+1} = \bbM_3(\bbu_{i,j-1}^n) - \bbM_4(\bbu_{i,j+1}^n), \\
	 \bbw_{1,i,j}^{n+1}& = \bbM_1(\bbu_{i-1,j}^n) + \bbM_2(\bbu_{i+1,j}^n), 
	\quad \bbw_{2,i,j}^{n+1} = \bbM_3(\bbu_{i,j-1}^n) +\bbM_4(\bbu_{i,j+1}^n).
\end{split}
\end{equation*}
The second-order scheme reads
\begin{equation*}
\begin{split}
	 \bbu_{i,j}^{n+1}& = \bbu_{i,j}^n - \frac{\Delta t^n}{\Delta x} \left[ \bbF_{i+1/2,j}^{(2),n} - \bbF_{i-1/2,j}^{(2),n}  \right] - \frac{\Delta t^n}{\Delta x} \left[ \bbG_{i,j+1/2}^{(2),n} - \bbG_{i,j-1/2}^{(2),n}  \right], \\
	 \bbv_{1,i,j}^{n+1}& = 2 (\bbM_1(\bbu_{i-1,j}^n) - \bbM_2(\bbu_{i+1,j}^n)) - \bbu_{1,i-1,j}^n + \bbu_{2,i+1,j}^n, \\
	\bbv_{2,i,j}^{n+1} &= 2(\bbM_3(\bbu_{i,j-1}^n) - \bbM_4(\bbu_{i,j+1}^n)) - \bbu_{3,i,j-1}^n + \bbu_{4,i,j+1}^n, \\
	\bbw_{1,i,j}^{n+1} &= 2 (\bbM_1(\bbu_{i-1,j}^n) + \bbM_2(\bbu_{i+1,j}^n)) - \bbu_{1,i-1,j}^n - \bbu_{2,i+1,j}^n, \\
	 \bbw_{2,i,j}^{n+1} &= 2(\bbM_3(\bbu_{i,j-1}^n) + \bbM_4(\bbu_{i,j+1}^n)) - \bbu_{3,i,j-1}^n - \bbu_{4,i,j+1}^n.
	\end{split}
\end{equation*}
To introduce a blending that acts as a flux limiter
, we suggest the following:
\begin{equation*}
\begin{split}
	 \bbu_{i,j}^{n+1} &= \bbu_{i,j}^n - \frac{\Delta t^n}{\Delta x} \left[ \bbF_{i+1/2,j}^{n} - \bbF_{i-1/2,j}^{n}  \right] - \frac{\Delta t^n}{\Delta x} \left[ \bbG_{i,j+1/2}^{n} - \bbG_{i,j-1/2}^{n}  \right], \\
	  \bbF_{i+1/2,j}^n &= \bbF_{i+1/2,j}^{(1),n} + \theta_{i+1/2,j}^n \Delta \bbF_{i+1/2,j}^n,
	 \qquad \bbG_{i,j+1/2}^n = \bbG_{i,j+1/2}^{(1),n} + \theta_{i,j+1/2}^n \Delta \bbG_{i,j+1/2}^n, \\
	 \bbv_{1,i,j}^{n+1} &= \bbM_1(\bbu_{i-1,j}^n) - \bbM_2(\bbu_{i+1,j}^n) \nonumber \\
	& \qquad + \theta_{i-1/2,j}^n \left[ \bbM_1(\bbu_{i-1,j}^n) - \bbu_{1,i-1,j}^n \right] - \theta_{i+1/2,j}^n \left[ \bbM_2(\bbu_{i+1,j}^n) - \bbu_{2,i+1,j}^n \right], \\
	 \bbv_{2,i,j}^{n+1}& = \bbM_3(\bbu_{i,j-1}^n) - \bbM_4(\bbu_{i,j+1}^n) \nonumber \\
	& \qquad + \theta_{i,j-1/2}^n \left[ \bbM_3(\bbu_{i,j-1}^n) - \bbu_{3,i,j-1}^n \right] - \theta_{i,j+1/2}^n \left[ \bbM_4(\bbu_{i,j+1}^n) - \bbu_{4,i,j+1}^n \right], \\
	 \bbw_{1,i,j}^{n+1} &= \bbM_1(\bbu_{i-1,j}^n) + \bbM_2(\bbu_{i+1,j}^n) \nonumber \\
	& \qquad + \theta_{i-1/2,j}^n \left[ \bbM_1(\bbu_{i-1,j}^n) - \bbu_{1,i-1,j}^n \right] + \theta_{i+1/2,j}^n, \left[ \bbM_2(\bbu_{i+1,j}^n) - \bbu_{2,i+1,j}^n \right], \\
	 \bbw_{2,i,j}^{n+1} &= \bbM_3(\bbu_{i,j-1}^n) + \bbM_4(\bbu_{i,j+1}^n) \nonumber \\
	& \qquad + \theta_{i,j-1/2}^n \left[ \bbM_3(\bbu_{i,j-1}^n) - \bbu_{3,i,j-1}^n \right] + \theta_{i,j+1/2}^n \left[ \bbM_4(\bbu_{i,j+1}^n) - \bbu_{4,i,j+1}^n \right].
	\end{split}
\end{equation*}
Coming back to the kinetic distributions by left-multiplying this scheme by $\mathbf{Q}^{-1}$ leads to the suggested blended vectorial LBM, which reads, after some computations,
\begin{equation}\label{eq:blending_u1}
\begin{split}
	 \bbu_{1,i,j}^{n+1} &= \bbM_1(\bbu_{i-1,j}^n) + \theta_{i-1/2,j}^n \left[ \bbM_1(\bbu_{i-1,j}^n) - \bbu_{1,i-1,j}^n \right], \\
	 \bbu_{2,i,j}^{n+1} &= \bbM_2(\bbu_{i+1,j}^n) + \theta_{i+1/2,j}^n \left[ \bbM_2(\bbu_{i+1,j}^n) - \bbu_{2,i+1,j}^n \right],\\
	 \bbu_{3,i,j}^{n+1} &= \bbM_3(\bbu_{i,j-1}^n) + \theta_{i,j-1/2}^n \left[ \bbM_3(\bbu_{i,j-1}^n) - \bbu_{3,i,j-1}^n \right], \\
	 \bbu_{4,i,j}^{n+1} &= \bbM_4(\bbu_{i,j+1}^n) + \theta_{i,j+1/2}^n \left[ \bbM_4(\bbu_{i,j+1}^n) - \bbu_{4,i,j+1}^n \right],
	\end{split}
	\end{equation}
	\begin{equation}\label{eq:blending_u5}
	\begin{split}
	 \bbu_{5,i,j}^{n+1} &= \bbM_5(\bbu_{i,j}^n) - \theta_{i-1/2,j}^n \left[ \bbM_2(\bbu_{i,j}^n) - \bbu_{2,i,j}^n \right] - \theta_{i+1/2,j}^n \left[ \bbM_1(\bbu_{i,j}^n) - \bbu_{1,i,j}^n \right]  \\
	& \qquad - \theta_{i,j-1/2}^n \left[\bbM_4(\bbu_{i,j}^n) - \bbu_{4,i,j}^n \right] - \theta_{i,j+1/2}^n \left[\bbM_3(\bbu_{i,j}^n) - \bbu_{3,i,j}^n \right]. 
\end{split}
\end{equation}
Note that, except for the blending parameters, all the operations done in the update of $\bbu_5$ in Eq.~\eqref{eq:blending_u5} are localized in $(i,j)$. They can be absorbed in the final update of $\bbu_{i,j}$ as
\begin{equation}
\begin{split}
    \bbu_{i,j}^{n+1} & = \bbu_{1,i,j}^{n+1} + \bbu_{2,i,j}^{n+1} + \bbu_{3,i,j}^{n+1} + \bbu_{4,i,j}^{n+1} + \alpha \bbu_{i,j}^n  \\
    & - \theta_{i-1/2,j}^n \left[\bbM_2(\bbu_{i,j}^n) - \bbu_{2,i,j}^n \right] - \theta_{i+1/2,j}^n \left[ \bbM_1(\bbu_{i,j}^n) - \bbu_{1,i,j}^n \right]  \\
    & - \theta_{i,j-1/2}^n \left[ \bbM_4(\bbu_{i,j}^n) - \bbu_{4,i,j}^n \right] - \theta_{i,j+1/2}^n \left[ \bbM_3(\bbu_{i,j}^n) - \bbu_{3,i,j}^n \right]. \label{eq:blending_u}
    \end{split}
\end{equation}
Replacing \eqref{eq:blending_u5} by \eqref{eq:blending_u}, we see that neither $\bbu_5$ nor $\bbM_5$ actually appear in the scheme. This approach could help save memory space, thereby increasing efficiency.

\begin{rmk}
	One may argue that, compared to the 	blended scheme of~\eqref{eq:naive_blending}, one lose an interesting locality property with the blending of \eqref{eq:blending_u1}-\eqref{eq:blending_u5} because they involve blending parameters computed at different locations in the apparent collision phase. In fact, since the blending parameters themselves are generally not local in space, even the collision phase in the blending of \eqref{eq:naive_blending} would be nonlocal.
\end{rmk}

The next question is how to define the blending parameters $\theta_{i+1/2,j}^n$ and $\theta_{i,j+1/2}^n$. To see how these parameters affect the updated conserved \textit{a priori}, especially the positivity of density and pressure, we perform a convex decomposition of the solution $\bbu_{i,j}^{n+1}$. We first see that
\begin{equation*}
\begin{split}
	\bbu_{i,j}^{n+1}  = & \bbu_{i,j}^n - \frac{\Delta t^n}{\Delta x} \left[ \bbF_{i+1/2,j}^{(1),n} - \bbF_{i-1/2,j}^{(1),n} + \bbG_{i,j+1/2}^{(1),n} - \bbG_{i,j-1/2}^{(1),n} \right]  \\
	& - \frac{\Delta t^n}{\Delta x} \Big[ \theta_{i+1/2,j}^n \Delta \bbF_{i+1/2,j}^n - \theta_{i-1/2,j}^n \Delta \bbF_{i-1.2,j}^n  \\
	&  \qquad \qquad  + \theta_{i,j+1/2}^n \Delta \bbG_{i,j+1/2}^n - \theta_{i,j-1/2}^n \Delta \bbG_{i,j-1/2}^n  \Big],
	\end{split}
\end{equation*}
Then, using the convex decomposition of the first-order terms \eqref{eq:D2Q5_convexproof_2}, we have
\begin{equation}
	\bbu_{i,j}^{n+1} = (2\alpha -1) \bbu_{i,j}^n + \frac{1-\alpha}{2} \left[ \bbu_{i+1/2,j}^{(1)} + \bbu_{i-1/2,j}^{(2)} + \bbu_{i,j+1/2}^{(3)} + \bbu_{i,j-1/2}^{(4)} \right], \label{eq:convex_decomposition_blending}
\end{equation}
where (the $\bbu^*$ are defined in \eqref{eq:D2Q5_convexproof_3})
\begin{equation*}
\begin{split}
	 \bbu_{i+1/2,j}^{(1)} &= \bbu_{i+1/2,j}^* - \frac{\theta_{i+1/2,j}^n}{\kappa^n}  \Delta \bbF_{i+1/2,j}^n,
	 \qquad \bbu_{i+1/2,j}^{(2)} = \bbu_{i+1/2,j}^* + \frac{\theta_{i+1/2,j}^n}{\kappa^n} \Delta \bbF_{i+1/2,j}^n, \\
	 \bbu_{i,j+1/2}^{(3)}& = \bbu_{i,j+1/2}^* - \frac{\theta_{i,j+1/2}^n}{\kappa^n} \Delta \bbG_{i,j+1/2}^n, 
\qquad \bbu_{i,j+1/2}^{(3)} = \bbu_{i,j+1/2}^* + \frac{\theta_{i,j+1/2}^n}{\kappa^n} \Delta \bbG_{i,j+1/2}^n,
	\end{split}
\end{equation*}
with $\kappa^n = a^n(1-\alpha)/2$. If $1/2 \leq \alpha < 1$, Eq.~\eqref{eq:convex_decomposition_blending} is a convex decomposition of the update variables. Hence, if we can ensure that $\bbu_{i,j}^n$, $\bbu_{i+1/2,j}^{(1)}$, $\bbu_{i+1/2,j}^{(2)}$, $\bbu_{i,j+1/2}^{(3)}$ and $\bbu_{i,j+1/2}^{(4)}$ satisfy some bounds inside a convex domain (\textit{e.g} local and global min/max principles), then the updated solution will satisfy the same bounds. In the present work, we focus on the following bounds:
\begin{enumerate}
	\item Density bounds: we want to ensure that
		$\mu_{i,j}^n \leq \rho_{i,j}^{n+1} \leq \nu_{i,j}^n,$
	where $\mu_{i,j}$ and $\nu_{i,j}$ are respectively local or global minima and maxima. Several choices can be considered for these extrema: positivity preservation (PP), local maximum principle (LMP)  or relaxed local maximum principle (RLMP). These choices are discussed in Sec.~\ref{sec:theta_rho}. This leads to first estimations of the blending parameters named $\theta_{i+1/2,j}^{\rho,n}$ and $\theta_{i,j+1/2}^{\rho,n}$. 
	\item Pressure bounds: we want to ensure positivity of the pressure in $\bbu_{i,j}^{n+1}$. The constraints on the blending parameters satisfying this condition are detailed in Sec.~\ref{sec:theta_p}. This leads to a second estimation of the blending parameters $\theta_{i+1/2,j}^{p,n}$ and $\theta_{i+1/2,j}^{p,n}$.
\end{enumerate} 
In the end, we take $\theta_{i+1/2,j}^n = \min(1, \theta_{i+1/2,j}^{\rho,n}, \theta_{i+1/2,j}^{p,n})$ and similarly for $\theta_{i,j+1/2}^n$. 
Let us first focus on the density bounds. From now on, we assume that $1/2 \leq \alpha < 1$, ensuring the convexity of the decomposition in \eqref{eq:convex_decomposition_blending}.

\subsection{Flux limiter for density bounds}
\label{sec:theta_rho}

We first detail how conditions on the blending parameters can be obtained  to ensure LMP for density, with a definition of local extrema involving intermediate Riemann states. The construction of the blending parameters ensuring PP and RLMP for density will then be very similar.

\subsubsection{Local maximum principle (LMP) based on intermediate states}
\label{sec:density_LE}

Inspired by~\cite{Vilar2024}, we choose local extrema as
\begin{equation}
\begin{split}
	 \mu_{i,j}^n &= \min(\rho_{i,j}^n, \rho_{i-1/2,j}^*, \rho_{i+1/2,j}^*, \rho_{i,j-1/2}^*, \rho_{i,j+1/2}^*), \label{eq:LMP_minmax} \\
	 \nu_{i,j}^n &= \max(\rho_{i,j}^n, \rho_{i-1/2,j}^*, \rho_{i+1/2,j}^*, \rho_{i,j-1/2}^*, \rho_{i,j+1/2}^*).
	 \end{split}
\end{equation}
From \eqref{eq:convex_decomposition_blending}, we have the convex decomposition
$$\rho_{i,j}^{n+1} = (2\alpha-1) \rho_{i,j}^n + \frac{1-\alpha}{2} \left[ \rho_{i+1/2,j}^{(1)} + \rho_{i+1/2,j}^{(2)} + \rho_{i,j+1/2}^{(3)} + \rho_{i,j+1/2}^{(4)} \right],
$$
where
\begin{equation*}
\begin{split}
	 \rho_{i+1/2,j}^{(1)} &= \rho_{i+1/2,j}^* - \frac{\theta_{i+1/2,j}^n}{\kappa^n}  \Delta F_{i+1/2,j}^{\rho,n}, \quad \rho_{i+1/2,j}^{(2)} = \rho_{i+1/2,j}^* + \frac{\theta_{i+1/2,j}^n}{\kappa^n} \Delta F_{i+1/2,j}^{\rho,n}, \\
	 \rho_{i,j+1/2}^{(3)} &= \rho_{i,j+1/2}^* - \frac{\theta_{i,j+1/2}^n}{\kappa^n} \Delta G_{i,j+1/2}^{\rho,n},\quad \rho_{i,j+1/2}^{(3)} = \rho_{i,j+1/2}^* + \frac{\theta_{i,j+1/2}^n}{\kappa^n} \Delta G_{i,j+1/2}^{\rho,n},
		\end{split}
\end{equation*}
where $\Delta F^{\rho} = \Delta \bbF \cdot [1,0,0,0]^T$ and similarly for $\Delta G^\rho$.
By construction of the local extrema, $\rho_{i,j}^n \in [\mu_{i,j}^n, \nu_{i,j}^n]$. It is therefore sufficient to ensure that the intermediate states satisfy these local bounds, i.e.
\begin{equation*}
\begin{split}
	\mu_{i+1,j}^n&\leq \rho_{i+1/2,j}^{(1)} \leq \nu_{i+1,j}^n, \qquad 
	\mu_{i,j}^n \leq \rho_{i+1/2,j}^{(2)} \leq \nu_{i,j}^n, \\
	\mu_{i,j+1}^n& \leq \rho_{i,j+1/2}^{(3)} \leq \nu_{i,j+1}^n, \qquad \mu_{i,j}^n \leq \rho_{i,j+1/2}^{(4)} \leq \nu_{i,j}^n.
	\end{split}
\end{equation*}
We focus on $\rho^{(1)}$ and $\rho^{(2)}$ first. There are several cases depending on the sign of $\Delta F_{i+1/2,j}^{\rho,n}$. If $\Delta F_{i+1/2,j}^{\rho,n}=0$, then $\rho_{i+1/2,j}^{(1)}$ and $\rho_{i+1/2,j}^{(2)}$ satisfy the local bounds whatever the value of $\theta_{i+1/2,j}^n$: we can take $\theta_{i+1/2,j}^n=1$. Suppose now that $\Delta F_{i+1/2,j}^{\rho,n} > 0$. Since $\theta_{i+1/2,j}^n \geq 0$, the local bounds $\rho_{i+1/2,j}^{(1)} \geq \mu_{i+1,j}^n$ and $\rho_{i+1/2,j}^{(2)} \leq \nu_{i,j}^n$ are satisfied. Therefore, we only need to ensure the two remaining conditions, i.e.
$$\theta_{i+1/2,j}^n \leq \kappa^n\, \frac{\nu_{i+1,j}^n - \rho_{i+1/2,j}^*}{\Delta F_{i+1/2,j}^{\rho, n}}, \qquad \theta_{i+1/2,j}^n \leq \kappa^n\,  \frac{\rho_{i+1/2,j}^* - \mu_{i,j}^n}{\Delta F_{i+1/2,j}^{\rho,n}}.
\\$$
Conversely, if $\Delta F_{i+1/2,j}^{\rho,n} <0$, the local bounds $\rho_{i+1/2,j}^{(1)} \leq \nu_{i+1,j}^n$ and $\rho_{i+1/2,j}^{(2)} \geq \mu_{i,j}^n$ are satisfied. Therefore we need to ensure the two remaining conditions
$$\theta_{i+1/2,j}^n \leq \kappa^n \, \frac{ \rho_{i+1/2,j}^* - \mu_{i+1,j}^n }{-\Delta F_{i+1/2,j}^{\rho, n}}, \qquad \theta_{i+1/2,j}^n \leq \kappa^n \, \frac{ \nu_{i,j}^n - \rho_{i+1/2,j}^*}{-\Delta F_{i+1/2,j}^{\rho,n}}.
$$
The procedure for $\theta_{i,j+1/2}^n$ is similar. In the end, we obtain
\begin{equation*}
\begin{split}
	\theta_{i+1/2,j}^{\rho,n}& = 
	\begin{cases}
	1 & \mathrm{if}\ \Delta F_{i+1/2,j}^{\rho,n} = 0, \\
	\kappa^n \min(\nu_{i+1,j}^n - \rho_{i+1/2,j}^*, \rho_{i+1/2,j}^*-\mu_{i,j}^n)/ \Delta F_{i+1/2,j}^{\rho,n} & \mathrm{if}\ \Delta F_{i+1/2,j}^{\rho,n} >0, \\
	-\kappa^n \min(\rho_{i+1/2,j}^*-\mu_{i+1,j}^n, \nu_{i,j}-\rho_{i+1/2,j}^*)/\Delta F_{i+1/2,j}^{\rho,n} & \mathrm{if}\ \Delta F_{i+1/2,j}^{\rho,n} < 0,
	\end{cases} \\
	 \theta_{i,j+1/2}^{\rho,n} &= 
	\begin{cases}
	1 & \mathrm{if}\ \Delta G_{i,j+1/2}^{\rho,n} = 0, \\
	\kappa^n \min(\nu_{i,j+1}^n - \rho_{i,j+1/2}^*, \rho_{i,j+1/2}^*-\mu_{i,j}^n)/ \Delta G_{i,j+1/2}^{\rho,n} & \mathrm{if}\ \Delta G_{i,j+1/2}^{\rho,n} >0, \\
	- \kappa^n \min(\rho_{i,j+1/2}^*-\mu_{i,j+1}^n, \nu_{i,j}-\rho_{i,j+1/2}^*)/\Delta G_{i,j+1/2}^{\rho,n} & \mathrm{if}\ \Delta G_{i,j+1/2}^{\rho,n} < 0.
	\end{cases} 
	\end{split}
\end{equation*}

\subsubsection{Relaxed local maximum principle (RLMP)}

The LMP constraint may be too strong for practical applications, leading to an increase of artificial diffusion even in absence of numerical oscillations: remember that for the Euler equations, the only physical constraint is $\rho>0$. 
For this reason, inspired by \cite{Yang2009}, we propose the use of relaxed local bounds defined as
\begin{equation*}
\begin{split}
 \mu_{i,j}^n &= 0.999\min(\rho_{i,j}^n, \rho_{i-1/2,j}^*, \rho_{i+1/2,j}^*, \rho_{i,j-1/2}^*, \rho_{i,j+1/2}^*), \\
 \nu_{i,j}^n& = 1.001\max(\rho_{i,j}^n, \rho_{i-1/2,j}^*, \rho_{i+1/2,j}^*, \rho_{i,j-1/2}^*, \rho_{i,j+1/2}^*).
	\end{split}
\end{equation*}
These bounds are less strict and still guaranty $\rho>0$. The intermediate states are all in $[\mu_{i,j}^n, \nu_{i,j}^n]$ and all the conclusions drawn previously 
 remain the same. We only have to modify the expression of the bounds $\mu_{i,j}^n$ and $\nu_{i,j}^n$ in the blending parameters.

\subsubsection{Positivity preservation (PP)}

Here we only want to ensure $\rho_{i,j}^{n+1} >0$ provided that $\rho_{i,j}^n>0$. A similar analysis as in Sec.~\ref{sec:density_LE} leads to the following constraints on the blending parameters:
$$
\theta_{i+1/2,j}^{n} < \kappa^n \frac{\rho_{i+1/2,j}^*}{|\Delta F_{i+1/2,j}^{\rho,n}|}, \qquad \theta_{i,j+1/2}^{n} < \kappa^n \frac{\rho_{i,j+1/2}^*}{|\Delta G_{i,j+1/2}^{\rho,n}|}.
$$
Note that contrary to the LMP and RLMP cases, these inequalities are strict to avoid the zero-density case. To satisfy these constraints, we set, with $\epsilon=10^{-16}$,
$$\theta_{i+1/2,j}^{\rho,n} = \max \left( 0, \kappa^n \frac{\rho_{i+1/2,j}^*}{|\Delta F_{i,j+1/2}^{\rho,n}|} - \epsilon \right), \quad \theta_{i+1/2,j}^{\rho,n} = \max \left( 0, \kappa^n \frac{\rho_{i+1/2,j}^*}{|\Delta G_{i+1/2,j}^{\rho,n}|} - \epsilon \right).
$$
\subsubsection{Relationship between the density bounds}

By construction, note that all the intermediate state densities are positive. Hence, the PP criterion is less strict  than the RLMP one, which is in turn less restrictive than the LMP one: 
$$0 \leq \left( \theta_{i+1/2,j}^{\rho,n} \right)_{LMP} \leq \left( \theta_{i+1/2,j}^{\rho,n} \right)_{RLMP} \leq \left( \theta_{i+1/2,j}^{\rho,n} \right)_{PP},$$
and similarly for $\theta_{i,j+1/2}^{\rho,n}$. Therefore, we expect the LMP model to be more dissipative than the RLMP one, which is in turn more dissipative than the PP one.

\subsection{Pressure positivity}
\label{sec:theta_p}

Following~\cite{Wu2023}, the pressure positivity condition can be equivalently written as the constraint
\begin{equation}
	\bbu_{i,j}^{n+1} \cdot \bbw > 0, \qquad \bbw=[(w_1^2+w_2^2)/2, -w_1, -w_2, 1]^T, \qquad \forall w_1,w_2 \in \mathbb{R}.
	\label{eq:GQL}
\end{equation}
This comes from the fact that
\begin{equation*}
	\bbu \cdot \bbw = \frac{p}{\gamma-1} + \frac{1}{2} \rho \left[ (v_1 - w_1)^2 + (v_2-w_2)^2 \right]. 
\end{equation*}
Let us assume that at time $t^n$, $\bbu_{i,j}^n \cdot \bbw >0$ for any mesh point $(i,j)$. Then, following \eqref{eq:convex_decomposition_blending}, $\bbu_{i,j}^{n+1} \cdot \bbw>0$ is satisfied provided that $\bbu_{i+1/2,j}^{(1)} \cdot \bbw >0$, $\bbu_{i+1/2,j}^{(2)} \cdot \bbw >0$, $\bbu_{i,j+1/2}^{(3)} \cdot \bbw >0$ and $\bbu_{i,j+1/2}^{(4)} \cdot \bbw >0$. These conditions are satisfied if 
\begin{equation*}
	\theta_{i+1/2,j}^n < \kappa^n \frac{\bbu_{i+1/2,j}^* \cdot \bbw}{|\Delta \bbF_{i+1/2,j}^n \cdot \bbw|}, \qquad \theta_{i,j+1/2}^n < \kappa^n \frac{\bbu_{i,j+1/2}^* \cdot \bbw}{|\Delta \bbG_{i,j+1/2}^n \cdot \bbw|},
\end{equation*}
for any $\bbw$ defined in \eqref{eq:GQL}. Let us focus on the first condition regarding $\theta_{i+1/2,j}^n$. Similarly to the one-dimensional case in~\cite{Pampa2024}, we see that we have to minimize an expression of the form
$$
	\varphi(w_1, w_2) := \frac{\rho^*(w_1^2 + w_2^2)-2m_1^* w_1 - 2m_2^* w_2 + 2E^*}{|\Delta F^\rho (w_1^2 + w_2^2) - 2 w_1 \Delta F^{m_1} - 2 w_2 \Delta F^{m_2} + \Delta F^E |},
$$
where $m_1^*=\bbu^* \cdot [0,1,0,0]^T$, $m_2^*=\bbu^* \cdot [0,0,1,0]^T$, $E^* = \bbu^* \cdot [0,0,0,1]^T$, $\Delta F^{m_1} = \Delta \bbF \cdot [0,1,0,0]^T$, $\Delta F^{m_2} = \Delta \bbF \cdot [0,0,1,0]^T$ and $\Delta F^{E} = \Delta \bbF \cdot [0,0,0,1]^T$. Using homogeneity, we rewrite it as
$
	\varphi(\bbz) = \frac{\langle \bbz, \mathbf{B} \bbz \rangle}{|\langle \bbz, \mathbf{A} \bbz \rangle|},
$ 
with $\bbz = [z_1,z_2,z_3]^T$, $z_i\in \R$ and the symmetric matrices
$$
	\mathbf{A} = 
	\begin{bmatrix}
		\Delta F^\rho & 0 & -\Delta F^{m_1} \\
		0 & \Delta F^\rho & -\Delta F^{m_2} \\
		- \Delta F^{m_1} & -\Delta F^{m_2} & 2 \Delta F^E
	\end{bmatrix}
	, \qquad \mathbf{B} = 
	\begin{bmatrix}
		\rho^* & 0 & -m_1^* \\
		0 & \rho^* & -m_2^* \\
		-m_1^* & -m_2^* & 2E^*
	\end{bmatrix}.
$$
Note that since $\langle \bbz, \mathbf{B} \bbz \rangle=z_3^2 \bbu^* \cdot \big (\bbz/z_3^2\big ) >0$, we know that $\mathbf{B}$ is positive definite. Our problem is to minimize $\varphi(\bbz)$ for $\bbz \in \mathbb{R}^3 \setminus \{\mathbf{0}\}$. Following~\cite{Pampa2024}, this can be done by solving an eigenvalue problem as
$$\min_{\bbz \in \mathbb{R}^3\setminus \{\mathbf{0}\}} \varphi(\bbz) = \left( \rho(\mathbf{B}^{-1/2} \mathbf{A} \mathbf{B}^{-1/2})  \right)^{-1}.
$$
Furthermore, as demonstrated in \cite{Pampa2024b}, the eigenvalues of $\mathbf{B}^{-1/2} \mathbf{A} \mathbf{B}^{-1/2}$ can be computed analytically: we get
$
	 \left\{ \frac{\Delta F^\rho}{\rho^*}, \frac{-\gamma_1 \pm \sqrt{\Delta}}{2 \gamma_0} \right\}, 
$ 
with
$\gamma_0 = {m_1^*}^2 + {m_2^*}^2 - 2 \rho^* E^*$, $\gamma_1 = 2 (\Delta F^\rho E^* + \Delta F^E \rho^* - \Delta F^{m_1} m_1^* - \Delta F^{m_2} m_2^*) $ and 
\begin{equation*}
	\Delta = \gamma_1^2 -4\gamma_0 \left[(\Delta F^{m_1})^2+(\Delta F^{m_2})^2-2 \Delta F^\rho \Delta F^E \right] \geq 0.
\end{equation*}
We also note that since $\gamma_0 = -\det(\mathbf{B})/\rho^* <0$, the spectral radius of $\mathbf{B}^{-1/2} \mathbf{A} \mathbf{B}^{-1/2}$ is given by
\begin{equation*}
	\rho(\mathbf{B}^{-1/2} \mathbf{A} \mathbf{B}^{-1/2}) = \max \left( \frac{|\Delta F^\rho|}{\rho^*}, \frac{|\gamma_1|+\sqrt{\Delta}}{-2\gamma_0} \right). 
\end{equation*}  
We finally take
\begin{equation*}
	\theta_{i+1/2,j}^{p,n} = \max \left(0,  \frac{\kappa^n}{\rho(\mathbf{B}^{-1/2} \mathbf{A} \mathbf{B}^{-1/2})} - \epsilon \right),
\end{equation*}
where the parameter $\epsilon = 10^{-16}$ is introduced to ensure a strict inequality in order to avoid the critical case of vanishing pressure.

\section{Numerical results}

In this section, our numerical strategy is assessed on well-known one- and two-dimensional problems involving the compressible Euler equations for fluid dynamics. All the numerical simulations proposed below are performed with the D2Q5 model of Example \ref{ex:D2Q5} with $\alpha=1/2$. It is the lowest value ensuring Eq.~\eqref{eq:convex_decomposition_blending} to be a convex decomposition. Except for Sec.~\ref{sec:sinus_entropy} where a particular choice is done, the kinetic speed is updated at each iteration as
$$
	a^n = 4 \lambda^n, \qquad \lambda^n = \max_{i,j} \left( \sqrt{v_1^2+v_2^2} + \sqrt{\gamma p/\rho} \right)^n_{i,j},
$$
which satisfies the critical case of Bouchut's criterion~\eqref{eq:Bouchut_D2Q5}. The time step is updated accordingly as $\Delta t^n = \Delta x / a^n$, which leads to a common CFL condition $\lambda^n \Delta t^n/\Delta x=0.25$. Unless otherwise stated, the blending parameters are always chosen in order to satisfy pressure positivity and one of the following constraints for the density: PP, LMP or RLMP. For all the test cases, the initialization of distributions is done as $\bbu_{k,i,j}^0 = \bbM_k(\bbu_0(x_i,y_j))$. All the one-dimensional test cases are performed with the two-dimensional method with $4$ cells and periodicity in the $y$-direction, and with $v_2=0$. In all cases, $\gamma=1.4$ except for LeBlanc and the astrophysical jets where $\gamma=5/3$, and the near vacuum smooth isentropic wave where $\gamma=3$.

\subsection{Sinusoidal entropy wave}
\label{sec:sinus_entropy}

This first example aims at evaluating the impact of the blending parameter on the numerical accuracy of the LB scheme in the case of a simple, linear, transport of an entropy wave in two dimensions. We consider a domain of size $[0,1]\times[0,1]$ with periodic boundaries and the initial condition
\begin{equation*}
	(\rho, v_1, v_2, p)_0(x,y) = (1 + 0.1 \sin(2\pi(x+y)), 1, 1, 1).
\end{equation*}
The solution is $\bbu(x,y,t)=\bbu_0(x-t,y-t)$. Numerical simulations are performed for different mesh sizes $\Delta x$ and for the three types of density bounds proposed in Sec.~\ref{sec:theta_rho} (PP, LMP and RLMP), until the final time $t=1$, which corresponds to one period of the sinusoidal wave in the square domain. For this test case, the kinetic speed is chosen as  $a^n = 10.7$. This ensures that $a^n \geq 4 \lambda^n$ and that the final time $t=1$ is a multiple of the time step for all the considered meshes. The $L^1$-, $L^2$ and $L^\infty$ errors  are reported in Table~\ref{tab:Orders_sin_entropy}.

\begin{table}[!ht]
\begin{subtable}{1.\textwidth}
\centering
\begin{tabular}{|p{0.07\textwidth}||>{\centering}p{0.16\textwidth}>{\centering}p{0.05\textwidth}|>{\centering}p{0.16\textwidth}>{\centering}p{0.05\textwidth}|>{\centering}p{0.16\textwidth}>{\centering\arraybackslash}p{0.05\textwidth}|}
\hline
$\Delta x$ & $L^1$-error & rate & $L^2$-error & rate & $L^\infty$-error & rate \\
\hline
$1/20$ & $3.23472\ 10^{-3}$ & $-$ & $3.57255\ 10^{-3}$ & $-$ & $4.57474\ 10^{-3}$ & $-$ \\ 
$1/40$ & $7.94386\ 10^{-4}$ & $2.03$ & $8.79313\ 10^{-4}$ & $2.02$ & $1.13088\ 10^{-3}$ & $2.02$ \\ 
$1/80$ & $1.97593\ 10^{-4}$ & $2.01$ & $2.18980\ 10^{-4}$ & $2.01$ & $2.82227\ 10^{-4}$ & $2.00$ \\ 
$1/160$ & $4.93618\ 10^{-5}$ & $2.00$ & $5.46922\ 10^{-5}$ & $2.00$ & $7.04894\ 10^{-5}$ & $2.00$ \\ 
$1/320$ & $1.23379\ 10^{-5}$ & $2.00$ & $1.36698\ 10^{-5}$ & $2.00$ & $1.76181\ 10^{-5}$ & $2.00$ \\ 
$1/640$ & $3.08429\ 10^{-6}$ & $2.00$ & $3.41724\ 10^{-6}$ & $2.00$ & $4.40429\ 10^{-6}$ & $2.00$ \\ 
\hline
\end{tabular}
\caption{PP limiter}
\end{subtable} \\
\begin{subtable}{1.\textwidth}
\centering
\begin{tabular}{|p{0.07\textwidth}||>{\centering}p{0.16\textwidth}>{\centering}p{0.05\textwidth}|>{\centering}p{0.16\textwidth}>{\centering}p{0.05\textwidth}|>{\centering}p{0.16\textwidth}>{\centering\arraybackslash}p{0.05\textwidth}|}
\hline
$\Delta x$ & $L^1$-error & rate & $L^2$-error & rate & $L^\infty$-error & rate \\
\hline
$1/20$ & $6.90503\ 10^{-3}$ & $-$ & $7.74117\ 10^{-3}$ & $-$ & $1.21124\ 10^{-2}$ & $-$ \\ 
$1/40$ & $2.13413\ 10^{-3}$ & $1.69$ & $2.68948\ 10^{-3}$ & $1.53$ & $5.13947\ 10^{-3}$ & $1.24$ \\ 
$1/80$ & $6.80204\ 10^{-4}$ & $1.65$ & $9.48031\ 10^{-4}$ & $1.50$ & $2.23845\ 10^{-3}$ & $1.20$ \\ 
$1/160$ & $2.02793\ 10^{-4}$ & $1.75$ & $3.25638\ 10^{-4}$ & $1.54$ & $9.30885\ 10^{-4}$ & $1.27$ \\ 
$1/320$ & $5.86654\ 10^{-5}$ & $1.79$ & $1.07246\ 10^{-4}$ & $1.60$ & $3.87212\ 10^{-4}$ & $1.27$ \\ 
$1/640$ & $1.77588\ 10^{-5}$ & $1.72$ & $3.50599\ 10^{-5}$ & $1.61$ & $1.65626\ 10^{-4}$ & $1.23$ \\ 
\hline
\end{tabular}
\caption{LMP limiter}
\end{subtable} \\
\begin{subtable}{1.\textwidth}
\centering
\begin{tabular}{|p{0.07\textwidth}||>{\centering}p{0.16\textwidth}>{\centering}p{0.05\textwidth}|>{\centering}p{0.16\textwidth}>{\centering}p{0.05\textwidth}|>{\centering}p{0.16\textwidth}>{\centering\arraybackslash}p{0.05\textwidth}|}
\hline
$\Delta x$ & $L^1$-error & rate & $L^2$-error & rate & $L^\infty$-error & rate \\
\hline
$1/20$ & $3.49578\ 10^{-3}$ & $-$ & $4.04074\ 10^{-3}$ & $-$ & $5.21913\ 10^{-3}$ & $-$ \\ 
$1/40$ & $1.04876\ 10^{-3}$ & $1.74$ & $1.29517\ 10^{-3}$ & $1.64$ & $2.32756\ 10^{-3}$ & $1.16$ \\ 
$1/80$ & $3.15374\ 10^{-4}$ & $1.73$ & $3.89949\ 10^{-4}$ & $1.73$ & $7.91117\ 10^{-4}$ & $1.56$ \\ 
$1/160$ & $7.58163\ 10^{-5}$ & $2.06$ & $9.32482\ 10^{-5}$ & $2.06$ & $1.96402\ 10^{-4}$ & $2.01$ \\ 
$1/320$ & $1.23694\ 10^{-5}$ & $2.62$ & $1.36899\ 10^{-5}$ & $2.77$ & $1.79201\ 10^{-5}$ & $3.45$ \\ 
$1/640$ & $3.08429\ 10^{-6}$ & $2.00$ & $3.41724\ 10^{-6}$ & $2.00$ & $4.40429\ 10^{-6}$ & $2.02$ \\ 
\hline
\end{tabular}
\caption{RLMP limiter}
\end{subtable}
\caption{Errors and convergence rate obtained by the blended LB schemes on the sinusoidal entropy wave.}
\label{tab:Orders_sin_entropy}
\end{table}

For the PP scheme, a second order of accuracy is obtained, as we do not need any non linear limiter to guaranty the positivity of $\rho$ and $p$. With the LMP scheme, a decrease of the order is observed, especially when the error is computed in the $L^\infty$-norm. Indeed,  $\theta < 1$ near the extrema of the entropy wave because of the strict choice of extrema in Eq.~\eqref{eq:LMP_minmax} for reducing numerical oscillations. Relaxing this constraint with the RLMP limiter leads to a decrease in the numerical errors.

\subsection{Sod shock tube}
In the second test case, we validate the ability of the method to deal with shock waves and contact discontinuities. We consider the Sod shock tube~\cite{Sod1978} defined on a one-dimensional domain of size $[0,1]$. This test case is computed with the PP, LMP and RLMP schemes with $N=100$ points in the $x$-direction and zero-gradient of all the distribution functions at the left and right boundaries. Moreover, to avoid spurious reflections, a sponge zone is introduced by setting $\theta=0$ in a layer of $5$ cells on the right and left boundaries.
 
\begin{figure}[h!]
	\centering
	\includegraphics{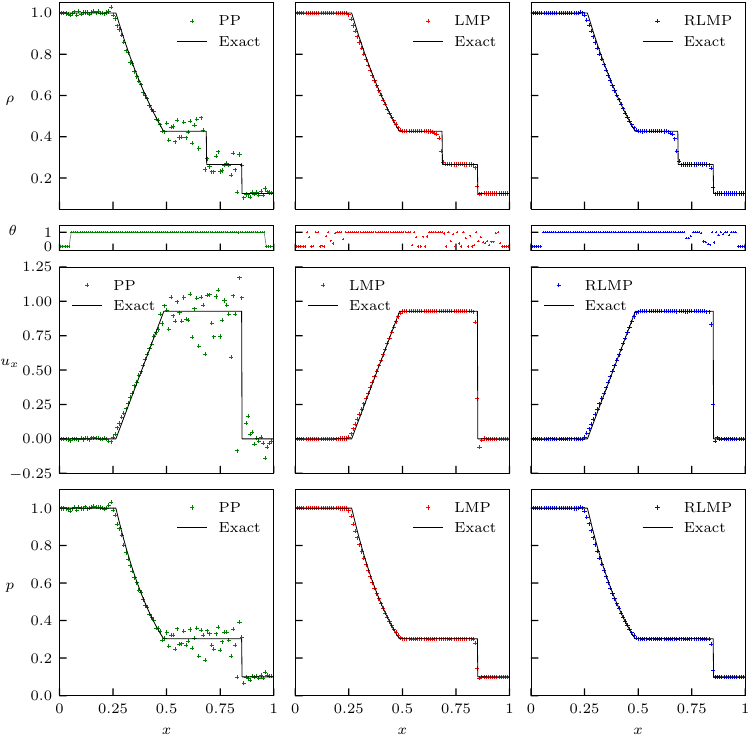}
        \caption{Sod shock tube at $t=0.2$ with $N=100$ points. Left: PP limiter, middle: LMP limiter, right: RLMP limiter.}
        \label{fig:Sod}
\end{figure}

The density, velocity, pressure and blending parameter profiles obtained at $t=0.2$ are displayed in Fig.~\ref{fig:Sod}. The PP limiter leads to very strong oscillations in a wide region between the rarefaction wave and the shock: the blending parameter is almost always set to $1$ in this case. Hence, the behavior of the second-order scheme, which is known to be very oscillating in presence of discontinuities~\cite{Graille2014, Dubois2014}, is recovered. For $N=100$, the density and the pressure remain positive even with the pure second-order scheme, this is however lost if one increases the resolution: a blending parameter is mandatory.
The LMP limiter allows to efficiently suppress the numerical oscillations, while maintaining remarkably sharp discontinuities. The shock and contact discontinuities are respectively captured on two and four points. Looking at the distribution of the blending parameter, we see that it is locally lower than $1$ at some points out of the discontinuities, which is how numerical oscillations can be suppressed in smooth regions. The RLMP limiter presents very few differences with the LMP one on this case: only very small oscillations are observed in smooth regions (\textit{e.g} close to the left boundary), which is due to larger admissible values for $\theta$.

\begin{figure*}[h!]
	\centering
	\includegraphics{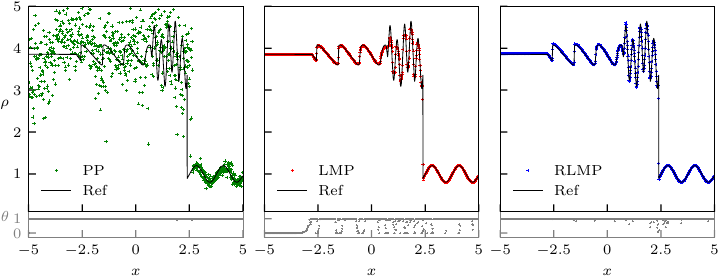}
	\caption{Density and blending parameter ($\theta$) profiles obtained for the Shu-Osher problem~\cite{Shu1988} at $t=1.8$ with $N=800$ points. Left: PP limiter, middle: LMP limiter, right: RLMP limiter. Reference: first-order scheme with $N=100000$ points.}
	\label{fig:ShuOsher_rho}
\end{figure*}

\subsection{Shu-Osher problem}

We consider the one-dimensional case of shock and sinusoidal wave interaction from Shu and Osher~\cite{Shu1988}. A domain of size $[-5,5]$ is initialized with the following conditions:
\begin{equation}
	(\rho, v_1, p)_0 =
	\begin{cases}
		(3.857143, 2.629369, 10.33333333333) \qquad \mathrm{if}\ x < -4, \\
		(1 + 0.2\, \sin(5x), 0, 1) \qquad \mathrm{else}.
	\end{cases} 
\end{equation}
This test case is simulated with the proposed blended LB schemes on a uniform mesh with $N=800$ points in the $x$-direction with zero-gradient conditions on the left and right boundaries. Density profiles obtained with the PP, LMP and RLMP limiters at time $t=1.8$ are displayed in Fig.~\ref{fig:ShuOsher_rho}. They are compared with a reference case obtained with the first-order scheme (equivalent to Lax-Friedrichs) with $N=100000$ points. As with the Sod Shock tube, the PP limiter leads to very large oscillations. Yet, density and pressure remain positive thanks to the introduction of the blending parameter, which can be locally lower than $1$ (see for example close to $x=2.5$ on Fig.~\ref{fig:ShuOsher_rho}). Note that in absence of limiter, the second-order scheme rapidly yields negative density and pressure on this case. With the LMP limiter, the numerical oscillations are remarkably reduced in all the domain, leading to a much more acceptable result. However, a major drawback of the LMP limiter is evidenced on this case: due to strict local constraints, the blending parameter switches to zero near all the extrema of the sinusoidal wave, thereby trimming the density profile. This leads to an over-dissipation especially observed downstream the shock wave. Relaxing the local constraints with the RLMP limiter leads to much more accurate results in line with the reference case. Velocity and pressure profiles obtained with the RLMP scheme, shown in Fig.~\ref{fig:ShuOsher_ux_p}, confirm the good agreement with the reference solution.

\begin{figure*}
	\centering
	\includegraphics{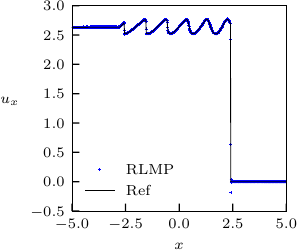} \hspace{1cm}
	\includegraphics{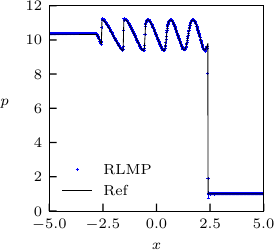}
	\caption{Velocity and pressure profiles of the Shu-Osher problem obtained with the relaxed local maximum principle (RLMP) at $t=1.8$ with $N=800$ points.}
	\label{fig:ShuOsher_ux_p}
\end{figure*}

\subsection{Near vacuum smooth isentropic wave}

We now consider a challenging smooth case involving very small values of pressure and density. 
A periodic domain of dimension $[-1,1]$ is considered with the initial conditions
$$
	(\rho, v_1, p)_0 = (1+\alpha \sin(\pi x), 0, \rho^\gamma), \quad \alpha=0.9999995 .
$$
When $\gamma=3$, the Riemann invariants follow two Burger's equations and the exact solution can be computed by the method of characteristics. Pressure and density are locally so close to zero that any scheme that does not ensure positivity preservation would fail this test (if $\alpha=1$, we get vacuum). In Fig.~\ref{fig:IsentropicSin_rho} are displayed the density profiles obtained at $t=0.1$ with the first-order scheme and the three assessed blended schemes with $N=50$ points. The PP limiter provides a very good agreement with the exact solution, especially when compared with the first-order scheme. The absence of numerical oscillations compared to the previous simulations is explained by the smoothness of this case, while numerical oscillations of the second-order LB scheme are generally triggered by discontinuities. With the LMP limiter, the solution is slightly deteriorated near the maximum of density, in a similar fashion that what was observed in~\cite{Vilar2024} with a discontinuous Galerkin scheme. With the proposed RLMP limiter, this defect is partly reduced, although not completely fixed.

\begin{figure*}
	\centering
	\includegraphics{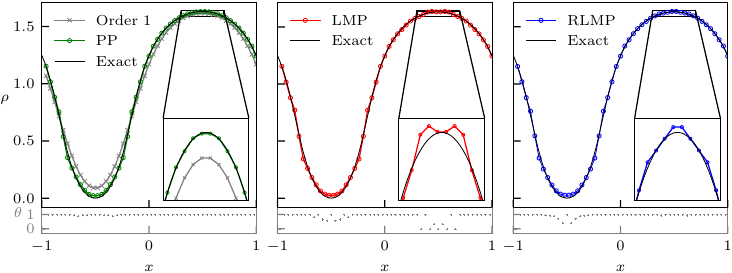}
	\caption{Density profiles of the near vacuum smooth isentropic wave at time $t=0.1$ with $N=50$ points.}
	\label{fig:IsentropicSin_rho}
\end{figure*}
%
%

\subsection{LeBlanc shock tube}

On the domain $[0,9]$, the initial conditions are given by:
\begin{equation*}
    \big (\rho, v_1, p\big )_0=\left\{\begin{array}{ll}
         \big (1, 0, 0.1(\gamma-1)\big ) &\mbox{if}~x \leq 3,\\
         \big (0.001, 0, 10^{-7}(\gamma-1)\big ) &\mbox{else}.
    \end{array}\right.
\end{equation*}
This is a very strong shock tube problem, where the correct location of the shock is difficult to obtain without enough points. This is a well known feature of this problem. The solution at time $t=6$ is displayed on Fig. \ref{fig:LeBlanc_rho} for the density and the blending parameter, and Fig. \ref{fig:LeBlanc_ux_p} for the velocity and the pressure with the RLMP limiter. The exact solution is also represented. In all cases, the positivity of the density and pressure are guaranted as expected. The PP blending is more oscillatory between the shock and the contact discontinuity: $\theta$ is almost everywhere equal to $1$. This problem is completely cured by the two other blending strategies.

\begin{figure*}
	\centering
	\includegraphics{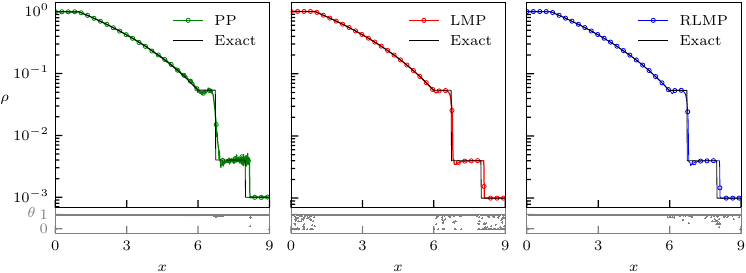}
	\caption{Density and blending parameter profiles of the LeBlanc problem at $t=6$ with $N=1000$ points. Left: PP limiter, middle: LMP limiter, right: RLMP limiter.}
	\label{fig:LeBlanc_rho}
\end{figure*}

\begin{figure*}
	\centering
	\includegraphics{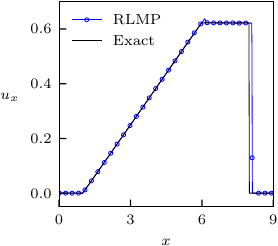}
	\hspace{1cm}
	\includegraphics{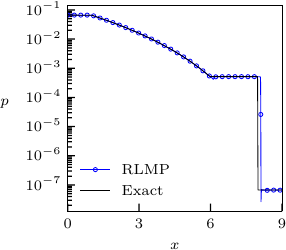}
	\caption{Velocity and pressure profiles of LeBlanc problem obtained with the RLMP limiter at $t=6$ with $N=1000$ points.}
	\label{fig:LeBlanc_ux_p}
\end{figure*}

\subsection{Sedov blast wave}

We consider the Sedov problem \cite{Sedov}, which features very low density with strong shocks. The domain is $[-1.1,1.1]^2$ and the initialization is done as follows: 
\begin{itemize}
     \item For all cells, we set  $\big(\rho, \bbv,  p\big)_0=\big ( 1,\mathbf{0},\frac{e_{\min}}{(\gamma-1)}\big )$ and $e_{\min}=10^{-12}$;
    \item Except the cell centered around $(0,0)$, we set $\big(\rho, \bbv, e\big)_0=\big(1,\mathbf{0}, \frac{e_{\max}}{\gamma-1}\big)$ with $e_{max}=0.979264/(\Delta x^2)$.
\end{itemize}
These conditions are set so that the strength of the shock is almost infinite, and the initial condition corresponds to a Dirac mass or energy at $\bbx=0$ in the limit of mesh refinement. The exact solution is computed following \cite{Kamm2000} and is evaluated at $t=1$ where it is known that the shock is a circle of radius $1$.

\begin{figure*}
	\centering
	\begin{subfigure}[b]{0.65\textwidth}
	\includegraphics[scale=0.90]{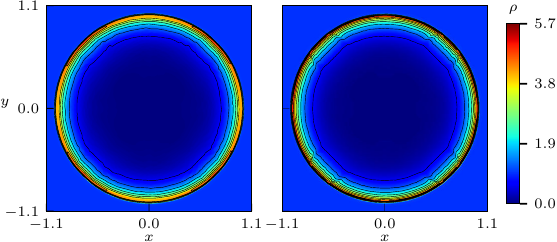}
	\end{subfigure}
	\begin{subfigure}[b]{0.33\textwidth}
		\includegraphics[scale=0.90]{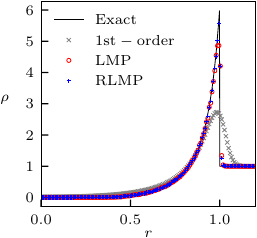}
	\end{subfigure}
	\caption{Sedov blast wave in two dimension at time $t=1$ on a $400\times 400$ cartesian mesh. Left: LMP limiter, middle: RLMP limiter, right: cutline along $y=x$.}
	\label{fig:sedov}
\end{figure*}

The solutions obtained with $400\times 400$ points are represented on figure \ref{fig:sedov}. Looking at the most right figure, one can see that the blending is very effective despite the stiffness of the problem: the solution is much better than with the first-order scheme, and very close to the exact one.

\subsection{Two-dimensional Riemann problem}

We consider a two-dimensional domain of size $[0,1]\times [0,1]$ with the initial condition of configuration 3 in~\cite{Lax1998}:
\begin{equation*}
	(\rho, v_1, v_2, p)_0 = 
	\begin{cases}
		(0.138, 1.206, 1.206, 0.029) & \mathrm{if}\ x \leq 0.8,\ y \leq 0.8, \\
		(0.5323, 1.206, 0, 0.3)  & \mathrm{if}\ x \leq 0.8,\ y > 0.8, \\
		(0.5323, 0, 1.206, 0.3) & \mathrm{if}\ x > 0.8,\ y \leq 0.8, \\
		(1.5, 0, 0, 1.5) & \mathrm{if}\ x > 0.8,\ y > 0.8, \\
	\end{cases}
\end{equation*}

Without any limiting, the simulation quickly crashes because of negative pressure and density. With the PP limiter, even though the simulation remains stable, strong oscillations makes the solution completely irrelevant. As shown in Fig.~\ref{fig:KT} displaying contours of density obtained at $t=0.8$, the LMP and RLMP limiters provide qualitatively accurate solution with sharp shocks and contact discontinuities, leading to Kelvin-Helmholtz instabilities even with $(400\times 400)$ points. Compared to LMP, the RLMP limiter reduces the numerical dissipation at the cost of acceptable oscillations.

\begin{figure*}	
\centering
\includegraphics[scale=0.97]{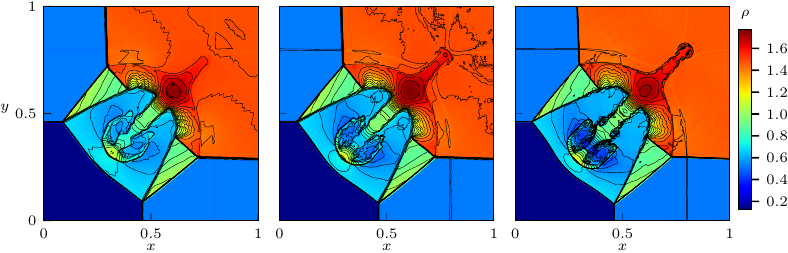}
\caption{Two dimensional Riemann problem (configuration 3 from~\cite{Lax1998}) at $t=0.8$. Left: LMP limiter with $(400 \times 400)$ points, middle: RLMP limiter with $(400 \times 400)$ points, right: RLMP limiter with $(800 \times 800)$ points. 30 equally-spaced contours lines from $0.127$ to $1.774$.}\label{fig:KT}
\end{figure*}

\subsection{Double Mach reflection}

We consider a two-dimensional domain of size $[0,3]\times [0,1]$ together with an oblique Mach 10 shock wave with the following pre- and post- shock states,
\begin{equation*}
	(\rho, v_1, v_2, p)_0 = 
	\begin{cases}
		(1.4, 0, 0, 1) & \mathrm{if}\ x \geq 1/6+ (y+20t)/\sqrt{3}, \\
		(8, 8.25\cos(\pi/6), -8.25\sin(\pi/6), 116.5) & \mathrm{else}.
	\end{cases}
\end{equation*}
These states are considered for the initial and boundary conditions (by setting the distribution to the corresponding Maxwellian) of the domain, except for the bottom boundary when $x \geq 1/6$, where a reflective (symmetry) wall is considered. As shown in fig.~\ref{fig:DMR} the LMP and RLMP are able to provide robust simulations with sharp discontinuities, giving birth to Kelvin-Helmoltz instabilities and controled numerical oscillations. 

\begin{figure*}
	\centering
	\includegraphics[scale=1]{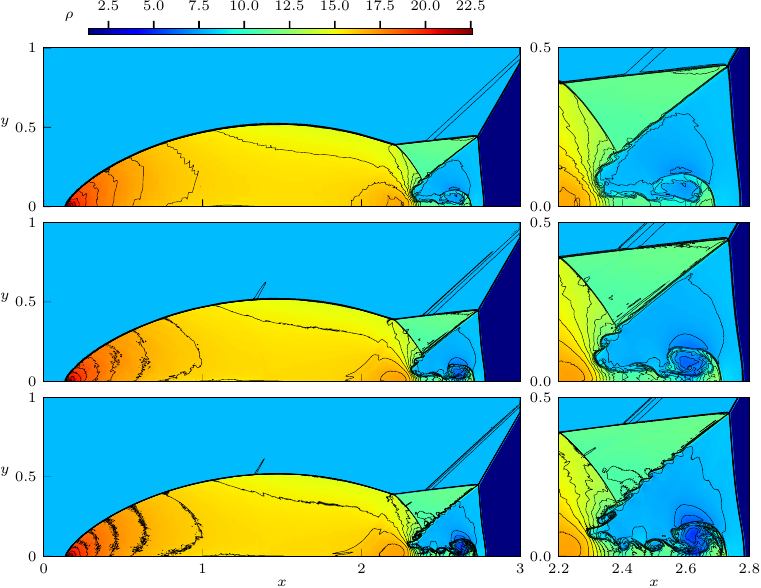}
	\caption{Density plots of the double Mach reflection at $t=0.2$. Top: LMP limiter with $(1500 \times 500)$ points, middle: RLMP limiter with $(1500 \times 500)$ points, bottom: RLMP limiter with $(3000 \times 1000)$ points. 30 equally-spaced contours from 1.386 to 22.589.}
	\label{fig:DMR}
\end{figure*}

\subsection{High Mach number astrophysical jets}

We conclude the numerical validations with two very challenging test cases in two dimensions: two astrophysical jets of Mach numbers 80 and 2000, following~\cite{Zhang2010}. In the first case, a computational domain of dimensions $[0, 2] \times [-0.5, 0.5]$ is considered with a fluid initially at conditions $(\rho, v_1, v_2, p)=(0.5, 0, 0, 0.4127)$. We consider zero-gradients conditions at the top, bottom and right boundaries, and impose $(\rho, v_1, v_2, p) = (5, 30, 0, 0.4127)$ if $|y|<0.05$ at the left boundary, initial conditions otherwise. The setup of the second case is similar, except that the considered domain is of dimensions $[0, 1] \times [-0.25, 0.25]$ and the inflow condition is $(\rho, v_1, v_2, p) = (5, 800, 0, 0.4127)$ if $|y| < 0.05$. Fig.~\ref{fig:AstroJets} displays the logarithm of density obtained at time $t=0.07$ for Mach=80 and $t=0.001$ for Mach=2000. In both cases, the LMP and RLMP limiters lead to stable numerical simulations with weak oscillations. Note that the LMP limiter introduces numerical dissipation in the jet core which affects the main flow structure. With the RLMP limiter, very encouraging results are obtained, especially when compared to similar simulations performed with high-order methods~\cite{Zhang2010, Duan2024-2D}.

\begin{figure*}
	\centering
	\includegraphics[scale=0.99]{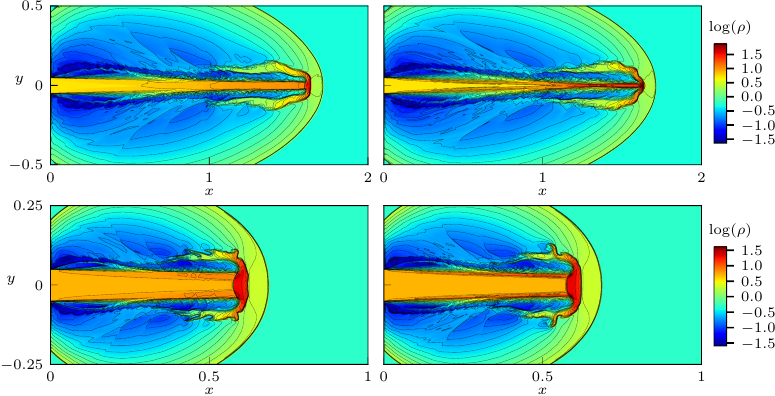}
	\caption{High Mach number astrophysical jets with $(800 \times 400)$ points. Top-left: Mach=80 with LMP limiter, top-right: Mach=80 with RLMP limiter, bottom-left: Mach=2000 with LBMP limiter, bottom-right: Mach=2000 with RLMP limiter.}
	\label{fig:AstroJets}
\end{figure*}

\section{Conclusion}
We have shown how to construct, from kinetic models inspired by~\cite{Jin, Natalini, Aregba} applied to the compressible Euler equations, a formally second-order scheme that preserves the positivity of the density and the internal energy. This is done thanks to a blending approach involving convex limiters carefully designed to (1) ensure conservation, (2) keep the simplicity of the collide and stream algorithm, and (3) preserve the solution in well-defined bounds. This strategy has been validated on challenging compressible test cases, providing guidance for selecting the RLMP limiter to achieve a balance between low numerical dissipation and spurious oscillations.

Though all the numerical tests we have performed use the perfect gas equation of state, the method can easily be generalized to more complex fluids. We have also shown how to use the collide-and-stream scheme in this context; this leads to a very efficient algorithm. Its performances will be reported in a future note.

There are many directions to explore now. One is how to extend this work to complex geometries. Another one is to adapt it to compressible viscous problems, following \cite{GauthierJCP1, GauthierJCP2}. More complex physics can also be considered, including for instance the introduction of source terms. The extension to 3D is in principle straightforward, and it will be interesting again to see the performance of the algorithm. We plan to address some or all of these issues in the near future.

\section*{Acknowledgements}

GW has been funded by SNFS grants \# 
200020\_204917 ``Structure preserving and fast methods for hyperbolic systems of conservation laws'' and FZEB-0-166980. The work of YL was supported by UZH Postdoc Grant, grant no. K-71120-05-01.

\appendix

\section{Multi-step finite difference interpretation of the second-order D2Q5 model}
\label{appendixA}

From \cite{Bellotti2022}, we show that the second-order scheme \eqref{eq:2nd-order_kinetic} with the D2Q5 model of Example~\ref{ex:D2Q5}, can be interpreted as the multi-step FD of \eqref{eq:multistep_FD} on the conserved variables. We start by rewriting the scheme as
\begin{equation}
    \begin{bmatrix}
        \bbu_1 \\ \bbu_2 \\ \bbu_3 \\ \bbu_4 \\ \bbu_5
    \end{bmatrix}_{i,j}^{n+1} =
    \mathbf{T}
    \begin{bmatrix}
        2 \bbM_1 - \bbu_1 \\ 2 \bbM_2 - \bbu_2 \\ 2 \bbM_3 - \bbu_3 \\ 2 \bbM_4 - \bbu_4 \\ 2 \bbM_5 - \bbu_5
    \end{bmatrix}_{i,j}^n, \qquad 
    \mathbf{T} = \mathrm{diag}(\bbx^-, \bbx^+, \bby^-, \bby^+, 1) \otimes \mathbf{I}_p,
    \label{eq:D2Q5_with_T}
\end{equation}
where $\otimes$ denotes the Kronecker product of two matrices and where $\bbx^-$, $\bbx^+$, $\bby^-$, $\bby^+$ are spatial operators defined such that
\begin{equation*}
    \bbx^- \bbu_{i,j}^n = \bbu_{i-1,j}^n, \quad \bbx^+ \bbu_{i,j}^n = \bbu_{i+1,j}^n, \quad \bby^- \bbu_{i,j}^n = \bbu_{i,j-1}^n, \quad \bby^+ \bbu_{i,j}^n = \bbu_{i,j+1}^n.
\end{equation*}
We then define the following moment matrix
\begin{equation}
	\mathbf{Q} = 
	\begin{bmatrix}
		1 & 1 & 1 & 1 & 1 \\
		1 & -1 & 0 & 0 & 0 \\
		0 & 0 & 1 & -1 & 0 \\
		1 & 1 & 0 & 0 & 0 \\
		0 & 0 & 1 & 1 & 0
	\end{bmatrix} \otimes \mathbf{I}_p , \quad \mathbf{Q}^{-1} =
	\begin{bmatrix}
		0 & 1/2 & 0 & 1/2 & 0 \\
		0 & -1/2 & 0 & 1/2 & 0 \\
		0 & 0 & 1/2 & 0 & 1/2 \\
		0 & 0 & -1/2 & 0 & 1/2 \\
		1 & 0 & 0 & -1 & -1
	\end{bmatrix} \otimes \mathbf{I}_p, \label{eq:moments_matrix}
\end{equation}
with which we define moments $
	\mathbf{m} := \mathbf{Q} [\bbu_1, \bbu_2, \bbu_3, \bbu_4, \bbu_5]^T$ and
\begin{equation*}
\begin{split}
	\mathbf{m}_{i,j}^{eq,n}& :=\mathbf{Q}[\bbM_1(\bbu_{i,j}^n), \bbM_2(\bbu_{i,j}^n), \bbM_3(\bbu_{i,j}^n), \bbM_4(\bbu_{i,j}^n), \bbM_5(\bbu_{i,j}^n)]^T \\
	& = [\bbu_{i,j}^n, \bbf(\bbu_{i,j}^n)/a^n, \bbg(\bbu_{i,j}^n)/a^n, (1-\alpha) \bbu_{i,j}^n/2, (1-\alpha) \bbu_{i,j}^n/2]^T.
	\end{split}
\end{equation*}
    Eq.~\eqref{eq:D2Q5_with_T} can be written in terms of moments as
\begin{equation}
	\mathbf{m}_{i,j}^{n+1} = \mathbf{A} \left[ \mathbf{m}_{i,j}^n - 2 \mathbf{m}_{i,j}^{eq,n} \right], \label{eq:monolithic_scheme_short}
\end{equation}
where 
\begin{align}
 \mathbf{A} = -\mathbf{Q} \mathbf{T} \mathbf{Q}^{-1} =
 \begin{bmatrix}
 	-1 & D_x & D_y & 1-A_x & 1-A_y \\
 	0 & -A_x & 0 & D_x & 0 \\
 	0 & 0 & -A_y & 0 & D_y \\
 	0 & D_x & 0 & -A_x & 0 \\
 	0 & 0 & D_y & 0 & -A_y
 \end{bmatrix} \otimes \mathbf{I}_p,
\end{align}
with $D_x = (\bbx^+-\bbx^-)/2$, $D_y=(\bby^+-\bby^-)/2$, $A_x=(\bbx^+ + \bbx^-)/2$ and $A_y=(\bby^+ + \bby^-)/2$. Following \cite{Bellotti2022}, we compute the characteristic polynomial of $\mathbf{A}$. It is given by:
\begin{align}
	P(z) = (z+1) (z^4 + 4A_a z^3 + (2+4A_d) z^2 + 4A_a z+1),
\end{align}
where $A_a$ and $A_d$  given by \eqref{eq:A_ad}. 
Using the Cayleigh-Hamilton theorem, 
$$
\left[\mathbf{A}^5 + (4A_a+1) (\mathbf{A}^4 + \mathbf{A}) + (4A_a+4A_d+2) (\mathbf{A}^3 + \mathbf{A}^2) + \mathbf{I} \right] \mathbf{m}_{i,j}^n = \mathbf{0}.
$$
Using \eqref{eq:monolithic_scheme_short}, the products $\mathbf{A}^k \mathbf{m}_{i,j}^n$ can be expressed as function of next time steps and equilibrium moments as 
$$
	\mathbf{A}^k \mathbf{m}_{i,j}^{n} = \mathbf{m}_{i,j}^{n+k} + 2 \sum_{l=1}^{k} \mathbf{A}^l \mathbf{m}_{i,j}^{eq,n+k-l}. 
$$
Plugging these expressions into $P(\mathbf{A})=0$ yields, after some calculations,
\begin{equation*}
\begin{split}
	& \mathbf{m}_{i,j}^{n+5} + (4A_a+1) (\mathbf{m}_{i,j}^{n+4} + \mathbf{m}_{i,j}^{n+1}) + (4A_a+4A_d+2)(\mathbf{m}_{i,j}^{n+3} + \mathbf{m}_{i,j}^{n+2}) \\
	& \quad  + \mathbf{m}_{i,j}^n - 4(2A_a+2A_d+1) \mathbf{m}_{i,j}^{eq,n+2} - 2(4A_a+1) \mathbf{m}_{i,j}^{eq,n+1} - 2 \mathbf{m}_{i,j}^{eq,n}  \\
	& \quad + 2 \mathbf{A} \mathbf{m}_{i,j}^{eq, n+4} + 2\left[ \mathbf{A}^2 + (4A_a+1)\mathbf{A} \right] \mathbf{m}_{i,j}^{eq,n+3}  \\
	& \quad - 2 \left[(4A_a+1) \mathbf{A}^{-1} + \mathbf{A}^{-2} \right] \mathbf{m}_{i,j}^{eq,n+2} - 2 \mathbf{A}^{-1}\mathbf{m}_{i,j}^{eq,n+1} = \mathbf{0}. 
	\end{split}
\end{equation*}
To obtain a multi-step finite-difference scheme over the conserved variables, we are interested in the first $p$ rows of this system. Noticing that
\begin{equation*}
\begin{split}
	\mathbf{r}_1 \mathbf{A} &= 
	\begin{bmatrix}
		-1 & D_x & D_y & 1-A_x & 1-A_y
	\end{bmatrix} \otimes \mathbf{I}_p, \\
	\mathbf{r}_1[\mathbf{A}^2+(4A_a+1)\mathbf{A}] &= 
	\begin{bmatrix}
		-4A_a & 3 \overline{D}_x & 3 \overline{D}_y & B & B
	\end{bmatrix} \otimes \mathbf{I}_p, \\
	\mathbf{r}_1 [(4A_a+1)\mathbf{A}^{-1} + \mathbf{A}^{-2}] &=
	\begin{bmatrix}
		-4A_a & -3 \overline{D}_x  & -3 \overline{D}_y & B & B
	\end{bmatrix} \otimes \mathbf{I}_p, \\
	\mathbf{r}_1 \mathbf{A}^{-1} &= 
	\begin{bmatrix}
		-1 & -D_x & -D_y & 1-A_x & 1-A_y 
	\end{bmatrix} \otimes \mathbf{I}_p,
	\end{split}
\end{equation*}
where $\mathbf{r}_1 = [1,0,0,0] \otimes \mathbf{I}_p$ and $B=4A_a-2A_d-A_x-1$, we obtain, after some computations, the scheme  \eqref{eq:multistep_FD}.

\bibliographystyle{siamplain}
\bibliography{refs}
\end{document}